\documentclass[A4paper,english,11pt]{article}
\usepackage{amsthm}
\usepackage[headings]{fullpage}

\usepackage[verbose]{wrapfig}
\usepackage{subfig,xcolor,graphicx}

\usepackage{amsmath,amssymb}

\newtheorem{thrm}{Theorem}
\newtheorem{lem}[thrm]{Lemma}
\newtheorem{cor}[thrm]{Corollary}

\newtheorem{conj}[thrm]{Conjecture}

\newcommand{\bvec}[1] {{\mathbf {#1}}}

\newcommand{\PCrit}{{P_{\rm cr}}}
\newcommand{\BPCrit}{{P_{\rm cr}}}

\newcommand{\TCrit}{{T_{\rm c}}}

\newcommand{\rmax}{ { {\it r_{\max}}  } }

\newcommand{\Schrodinger}{{Schr\"odinger }}

\newcommand{\mycaption}[1]{\parbox{0.95\textwidth}{\caption{{#1}}}}

\newcommand{\Real}{{\mathbb R}}
\newcommand{\abs}[1]{\left\vert{#1}\right\vert}
\newcommand{\norm}[1]{\left\Vert{#1}\right\Vert}

\newcommand{\weakto}{{\rightharpoonup}}

\newcommand{\LN}{{ \mathit l }}

\begin{document}

\renewcommand{\thesubfigure}{{\Alph{subfigure}}}

\title{Singular solutions of the biharmonic %
Nonlinear Schr\"odinger equation}

\makeatletter
\author{{G. Baruch}$^*$, {G. Fibich}, {E. Mandelbaum},\\
	School of Mathematical Sciences, Tel Aviv University, Tel Aviv 69978, Israel\\
	$^*$Corresponding author, guybar@math.tau.ac.il
}
\makeatother
\maketitle

\begin{abstract}
	We consider singular solutions of the biharmonic NLS.
	In the $L^2$-critical case, the blowup rate is bounded by a quartic-root
	power law, the solution approaches a self-similar profile, and a finite amount
	of $L^2$-norm, which is no less than the critical power, concentrates into
	the singularity (``strong collapse'').
	In the $L^2$-critical and supercritical cases, we use asymptotic analysis
	and numerical simulations to characterize singular solutions with a peak-type
	self-similar collapsing core.
	In the critical case, the blowup rate is slightly faster than a
	quartic-root, and the self-similar profile is given by the standing-wave
	ground-state.
	In the supercritical case, the blowup rate is exactly a quartic-root, and
	the self-similar profile is a zero-Hamiltonian solution of a nonlinear
	eigenvalue problem.
	These findings are verified numerically (up to focusing levels of~$10^8$)
	using an adaptive grid method.
	We also calculate the ground states of the standing-wave equations and
	the critical power for collapse in two and three dimensions.

\end{abstract}

\section{\label{sec:intro}Introduction}


The focusing nonlinear \Schrodinger equation (NLS)
\begin{equation}	\label{eq:NLS}
	i\psi_t(t,\bvec{x}) + \Delta\psi + \left|\psi\right|^{2\sigma}\psi = 0,
	\qquad \psi(0,\bvec{x}) = \psi_0(\bvec{x})\in H^1(\Real^d),
\end{equation}
where~$
	\bvec x=\left(x_1,\dots,x_d\right)\in\Real^d
$, and~$
	\Delta = \sum_{j=1}^d\partial_j^2
$ is the Laplacian, has been the subject of intense study, due to its role in
various areas of physics, such as nonlinear optics and Bose-Einstein Condensates
(BEC).
It is well-known that the NLS~\eqref{eq:NLS} possesses solutions that become
singular in a finite time~\cite{Sulem-99}.
Of special interest is the critical~($\sigma=2/d$) NLS
\begin{equation}	\label{eq:CNLS}
	i\psi_t(t,\bvec{x}) + \Delta\psi + \left|\psi\right|^{4/d}\psi = 0,
	\qquad \psi(0,\bvec{x}) = \psi_0(\bvec{x})\in H^1(\Real^d),
\end{equation}
which models the collapse of intense laser beams that propagate in a bulk Kerr
medium.

In this study, we consider the focusing biharmonic nonlinear \Schrodinger equation
(BNLS)
\begin{equation}	\label{eq:BNLS}
	i\psi_t(t,\bvec{x}) - \Delta^2\psi + \left|\psi\right|^{2\sigma}\psi = 0,
	\qquad \psi(0,\bvec{x}) = \psi_0(\bvec{x})\in H^2(\Real^d),
\end{equation}
where~$\Delta^2$ is the biharmonic operator.
Equation~\eqref{eq:BNLS} admits waveguide (standing-wave) solutions of the form~$
	\psi(t,\bvec{x}) =
		\lambda^{2/\sigma} e^{i\lambda^4t} R(\lambda\bvec{x})
$, where~$R$ satisfies the "standing-wave" equation 
\begin{equation}	\label{eq:stationary_state}
	-\Delta^2 R(\bvec{x}) - R + |R|^{2\sigma}R = 0.
\end{equation}

%
The BNLS~\eqref{eq:BNLS} is called~``{\em $L^2$-critical}'', or simply
``critical'' if~$\sigma d=4$.
In this case, the~$L^2$ norm (``{\em power}'') is conserved under the BNLS dilation
symmetry~$
	\psi\left( t,\bvec{x} \right)
	\mapsto L^{-2/\sigma} \psi\left(t/L^4,\bvec{x}/L \right)
$.
The critical BNLS can be rewritten as
\begin{equation}	\label{eq:CBNLS}
	i\psi_t(t,\bvec{x}) - \Delta^2\psi
		+\left|\psi\right|^{8/d}\psi  =  0,
	\qquad \psi(0,\bvec{x}) = \psi_0(\bvec{x})\in H^2(\Real^d).
\end{equation}
Correspondingly, the BNLS with~$\sigma d<4$ is called subcritical, and the BNLS 
with~$\sigma d>4$ is called supercritical.
This is analogous to the NLS, where the critical case is~$\sigma d=2$.

In~\cite{Ben-Artzi-00}, Ben-Artzi, Koch and Saut proved that the
BNLS~\eqref{eq:BNLS} is locally well-posed in~$H^2$, when~$\sigma$ is in
the \emph{$H^2$-subcritical} regime
\begin{equation}	\label{eq:admissible-range}
	\begin{cases}
		0<\sigma & d\leq4,\\
		0<\sigma<\frac4{d-4} & d>4.
	\end{cases}
\end{equation} 

\noindent
Global existence and scattering of BNLS solutions in the $H^2$-critical 
case~$\sigma=4/(d-4)$ were studied by Miao, Xu and Zhao~\cite{Miao20093715} and by
Pausader~\cite{Pausader_DCDS-2009a}.
The latter work also showed well-posedness for small data.
The $H^2$-critical defocusing BNLS was studied by Miao, Xu and Zhao~\cite{miao-2008}
and by Pausader~\cite{Pausader_DPDEs-2007,Pausader_JFA2009}.

The above studies focused on non-singular solutions.
In this work, we study singular solutions of the BNLS in~$H^2$, i.e., solutions that
exist in~$H^2(\Real^d)$ over some finite time interval~$t\in[0,\TCrit)$, 
but for which~$
	\displaystyle \lim_{t\to\TCrit} \norm{\psi}_{H^2} = \infty.
$
The first study of singular BNLS solutions was done by Fibich, Ilan and
Papanicolau~\cite{Fibich_Ilan_George_BNLS:2002}, who proved the following results:
\begin{thrm} \label{thrm:GE_subcritical}
	Solutions of the subcritical ($\sigma d<4$) focusing BNLS~\eqref{eq:BNLS} exist
	globally.
\end{thrm} 
\begin{thrm}	\label{thrm:GE_critical}
	Let $ \norm{\psi_0}_2^2 < \BPCrit$, where~$\BPCrit = \norm{R}_2^2$, and~$R$ is
	the ground state of~\eqref{eq:stationary_state} with~$\sigma=4/d$.
	Then, the solution of the critical focusing BNLS~\eqref{eq:CBNLS} exists
	globally.
\end{thrm} 

\noindent
The simulations in~\cite{Fibich_Ilan_George_BNLS:2002} suggested that there exist
singular solutions for~$\sigma d=4$ and~$\sigma d>4$, and that these singularities
are of the blowup type, namely, the solution becomes infinitely
localized.
However, in contradistinction with NLS theory, there is currently no rigorous proof
that solutions of the BNLS can become singular in either the critical or the
supercritical case.

To the best of our knowledge, the only work, apart
from~\cite{Fibich_Ilan_George_BNLS:2002}, which considered singular solutions of
the BNLS is by Chae, Hong and Lee~\cite{ChaeHongLee:2009}, who proved that if
singular solutions of the critical BNLS exist, then they have a
power-concentration property.
See Section~\ref{ssec:power_conc} for more details.

\subsection{\label{ssec:intro_summary}Summary of results}

In this work, we consider singular solutions of the focusing BNLS in the
$H^2$-subcritical regime~\eqref{eq:admissible-range}.
Our purpose is to characterize these singular solutions:
Their profile, blowup rate, power concentration, et cetera.

In some cases, we assume radial symmetry, i.e., that~$
	\psi(t,\bvec{x}) \equiv \psi(t,r)
$
where $
	r=\norm{\bvec x}_2=\sqrt{x_1^2+\dots+x_d^2}.
$
In these cases, equation~\eqref{eq:BNLS} reduces to 
\begin{equation}	\label{eq:radial_BNLS}
	i\psi_t(t,r) - \Delta^2_r\psi + \left|\psi\right|^{2\sigma}\psi = 0,
	\qquad \psi(0,r) = \psi_0(r),
\end{equation}
where
\begin{equation}	\label{eq:radial_bi_Laplacian}
	\Delta_r^2 = 
		\partial_r^4 
		+\frac{2(d-1)}{r}\partial_r^3
		+\frac{(d-1)(d-3)}{r^2}\partial_r^2
		-\frac{(d-1)(d-3)}{r^3}\partial_r
\end{equation}
is the radial biharmonic operator.
Specifically, the critical BNLS~\eqref{eq:CBNLS} reduces to 
\begin{equation}	\label{eq:radial_CBNLS}
	i\psi_t(t,r) - \Delta^2_r\psi + \left|\psi\right|^{8/d}\psi = 0,
	\qquad \psi(0,r) = \psi_0(r).
\end{equation}

The paper is organized as follows:
In Section~\ref{sec:Invariance}, we use Noether Theorem to derive conservation laws
for the BNLS.
We recall that in the critical NLS, the conservation law which follows from
invariance of the action integral under dilation leads to the NLS
``Variance Identity'', which can be used to prove the existence of singular
solutions.
In Section~\ref{ssec:varid} we use a similar procedure to derive the
``Variance Identity'' for the critical BNLS, and then generalize it to the
supercritical BNLS.
However, since it is not clear that the ``BNLS variance'' is positive definite, this
identity does not lead to a proof of the existence of singular solutions.

The ground states of the BNLS standing-wave equation~\eqref{eq:stationary_state}
were previously computed only in the one-dimensional
case~\cite{Fibich_Ilan_George_BNLS:2002}, since they were computed using a shooting
method, which cannot be easily generalized to multi-dimensions.
In Section~\ref{sec:stationary} we use the spectral renormalization method to
compute the ground-states of the critical BNLS~\eqref{eq:CBNLS} for one, two and
three dimensions.
The calculated ground-states provide the first numerical estimate of the critical
power for collapse $\PCrit=\norm{R}^2_2$ in the two-dimensional and
three-dimensional cases, see Section~\ref{sec:critical_power}.
Direct simulations of the critical BNLS suggest that the constant~$\PCrit$ in
Theorem~\ref{thrm:GE_critical} is optimal.

In Section~\ref{sec:rigorous} we use rigorous analysis to study the critical
BNLS~\eqref{eq:CBNLS}.
The blowup rate is shown to be lower-bounded by a quartic root, i.e.,
$
	\norm{\Delta\psi}_2^{-1/2} \leq C\left( \TCrit-t \right)^{1/4}.
$
The corresponding bound for the critical NLS is a square root, i.e.,
$
	\norm{\nabla\psi}_2^{-1} \leq C\left( \TCrit-t \right)^{1/2}.
$
We then prove that singular solutions converge to a self-similar profile
strongly in~$L^{2+2\sigma}$, for any~$\sigma$ in the $H^2$-subcritical
regime~\eqref{eq:admissible-range}.
This implies that the singular solutions have the power-concentration property,
whereby the amount of power that enters the singularity point is at least~$\PCrit$.
These rigorous results mirror those of the critical NLS.

Let us denote the location of the maximal amplitude of a radially-symmetric
solution by
\[	\rmax(t) = \displaystyle \arg \max_r|\psi|. 	\]
Singular solutions are called ``{\em peak-type}'' when $\rmax(t)\equiv0$
for~$0\le t \le T_c$, and ``{\em ring-type}'' when~$\rmax(t)>0$ for~$0\le t<\TCrit$.
In this work, we use asymptotic analysis and numerics to find and characterize
peak-type singular solutions of the BNLS equation.
Ring-type singular solutions of the BNLS will be studied
elsewhere~\cite{Baruch_Fibich_Gavish:2009,Baruch_Fibich_Mandelbaum:2009b}.

In Section~\ref{sec:crit_peak} we use asymptotics and numerics to show that
peak-type singular solutions of the critical BNLS collapse with the
quasi self-similar profile
\[
	\psi(t,r) \sim
		\frac{1}{L^{d/2}(t)}
		R\left( \frac r{L(t)} \right)
		e^{i\int^t_0\frac{1}{L^4(t')}dt'},
		\qquad \lim_{t\to\TCrit}L(t)=0,
\]
where the self-similar profile is the ground state~$R(\rho)$.
The blowup rate is shown to be slightly faster than the quartic-root bound.
This is analogous to the critical NLS, where the blowup rate of peak-type
solutions is slightly faster than the square-root bound, due to the loglog
correction (the ``loglog law'').
It is an open question whether the correction to the BNLS blowup rate is also a
loglog term.

In Section~\ref{sec:supercrit_peak} we consider peak-type singular solutions of the
supercritical BNLS.
In this case, asymptotics and numerics show that singular solutions are of the
quasi self-similar form
\begin{equation}	\label{eq:intro_SS_supercrit}
\begin{gathered}
	\psi(t,r) \sim 
		\frac{1}{L^{2/\sigma}(t)}
		B\left( \frac r{L(t)} \right)
		e^{i \int \frac{1}{L^4(t')}dt'}
		,\qquad
		\lim_{t\to\TCrit}L(t)=0,
\end{gathered}
\end{equation} 
the blowup rate of~$L(t)$ is exactly a quartic-root, and the self-similar
profile~$B(\rho)$ is different from the ground-state.
Rather, as in the supercritical NLS, the self-similar profile~$B$ is the
zero-Hamiltonian solution of a nonlinear eigenvalue problem.
Although $B(\rho)$ is not in~$L^2$, it can be the self-similar profile of a
collapsing~$H^2$ solution, since the collapsing solution is ``only'' quasi
self-similar.

Section~\ref{sec:num_meth} presents the numerical methods.
The computations of singular BNLS solutions that focus by factors of~$10^8$ and
more, necessitated the usage of adaptive grids.
We develop a modified version of the Static Grid Redistribution
method~\cite{Ren-00,SGR-08}, which is more convenient for the biharmonic problem.
Calculating the BNLS standing waves in multi-dimensions is done using the Spectral
Renormalization Method.

\subsection{\label{ssec:discussion}Discussion}

In this study, we use rigorous theory, asymptotic theory and numerics to analyze
singular solutions of the BNLS.
All the results presented in this work mirror those of the NLS, ``up to a change by
a factor of~$2$'' in the blowup rate~($1/2\longrightarrow1/4$), in the critical
value of~$\sigma$ ($2/d\longrightarrow4/d$), et cetera.
However, several key features of NLS theory are still missing from BNLS theory.
First, the ``Variance Identity'' for the BNLS cannot be used to prove that singular
solutions exist.
Second, the critical NLS is invariant under the pseudo-conformal
(``lens-transformation'') symmetry which can also be used to construct explicit
singular solutions.
At this time, it is unknown whether an analogous identity for the critical BNLS
exists.
Third, in critical NLS theory, the self-similar profile is known to possess a
quadratic radial phase term, i.e., \[
	\psi(t,r) \sim
	\frac{1}{L^{d/2}(t)}
	R\left( \frac r{L(t)} \right)
	e^{
		i\tau(t) +
		i \frac{L_t}{4L}r^2
	}.
\]
This term represents the focusing of the solution towards~$r=0$, and plays a key
role in the rigorous and asymptotic theory of the critical NLS.
At this time, we do not know the analogous radial phase term for the critical BNLS.

Finally, we note that a similar ``up to a factor of~$2$'' connection exists
between singular solutions of the nonlinear heat equation, see~\cite{GigaKohn_85},
and the biharmonic nonlinear heat equation,
see~\cite{Budd-Williams-Galaktionov_2004}.
For example, the~$L^\infty$ norm of singular solutions blows up as~$
	\left( \TCrit-t \right)^{-1/2}
$
for the nonlinear heat equation, and as~$
	\left( \TCrit-t \right)^{-1/4}
$ for the biharmonic heat equation.
The ``similarity up to a factor of~$2$'', however, is not perfect.
For example, the self-similar spatial variable is~$
	r/\sqrt{( \TCrit-t ) \log(\TCrit-t) }
$ for the nonlinear heat equation, and~$
	r/\sqrt[4]{ \TCrit-t }
$ for the biharmonic nonlinear heat equation.
Another difference between the equations is that singular solutions are
{\em asymptotically} self-similar for the nonlinear heat equation, and truly
self-similar for the biharmonic heat equation.
In contrast, the NLS possesses self-similar singular solutions, whereas for the BNLS
it is unknown whether singular solutions are truly, or only asymptotically,
self-similar.


\section{\label{sec:Invariance}Invariance}

The BNLS~\eqref{eq:BNLS} is the Euler-Lagrange equation of the action integral \[
	S = \int \mathcal{L} d\bvec{x} dt,
\] where~$\cal L$ is the  Lagrangian density
\begin{equation}	\label{eq:Lagrangian}
	\mathcal{L}\left(\psi, \psi^*, \psi_t, \psi_t^*,
		\Delta\psi, \Delta\psi^* \right)
		= \frac{i}{2} \left(
			\psi_t\psi^* - \psi_t^*\psi
		\right)
		- \left|\Delta \psi\right|^2
		+ \frac{1}{1+\sigma} \left|\psi\right|^{2(\sigma+1)}.
\end{equation}	
Therefore, the conserved quantities of the BNLS can be found using Noether
theorem, see Appendix~\ref{app:Noether}.
As in the standard NLS, invariance of the action integral under
phase-multiplications~$\psi\mapsto e^{i\delta}\psi$ implies conservation of the
``power'' (~$L^2$ norm ), i.e,. \[
	P(t) \equiv P(0),\qquad 
	P(t) = \norm{\psi(t)}_2^2.
\]
Similarly, invariance under temporal translations~$t\mapsto t+\delta t$ implies
conservation of the Hamiltonian
\begin{equation}	\label{eq:Hamiltonian_conservation}
	H(t) \equiv H(0),\qquad 
	H[\psi(t)]  =  
		\norm{\Delta\psi}_2^2 
			-\frac{1}{\sigma+1} \norm{\psi}_{2(\sigma+1)}^{2(\sigma+1)},
\end{equation}
and invariance under spatial translations~$\bvec{x}\mapsto\bvec{x}+\delta\bvec{x}$
implies conservation of the linear momentum, i.e., \[
	\bvec{P}(t) \equiv\bvec{P}(0),\qquad 
	\bvec{P}(t) = \int \text{Im}\left\{\psi^*\nabla\psi\right\}d\bvec{x}\,.
\]
In the critical case~$\sigma\cdot d=4$, the action integral is also invariant under
the dilation transformation~$
	\psi(t,\bvec{x}) \to \lambda^{d/2}\psi(\lambda^4 t,\lambda\bvec{x}).
$
The corresponding conserved quantity is 
\begin{equation}	\label{eq:J_dilation}
	J(t) \equiv J(0),\qquad 
	J(t) = \int \bvec{x} \cdot \text{Im}\left\{
			\psi^*\nabla\psi
		\right\}d\bvec{x} + 4 t H.
\end{equation}

\subsection{\label{ssec:varid}Towards a variance identity}

We recall that the action integral of the critical NLS~\eqref{eq:CNLS} is
invariant under the dilation transformation \[
	\psi_\text{NLS}(t,\bvec{x}) \mapsto
		\lambda^{d/2}\psi_\text{NLS}(\lambda^2 t,\lambda\bvec{x}) .
\]
The corresponding conserved quantity is 
\begin{equation}	\label{eq:J-dilation-NLS}
	J_\text{NLS}(t)
		\equiv J_\text{NLS}(0),\qquad 
	J_\text{NLS}
		= \int \bvec{x} \cdot \text{Im}\left\{
			\psi^*\nabla\psi
		\right\}d\bvec{x} - 2 t H_\text{NLS}\,.
\end{equation}
In addition, the integral term~$
	\int \bvec{x} \cdot \text{Im}\left\{ \psi^*\nabla\psi \right\}d\bvec{x} 
$ is the time-derivative of the variance, i.e., \[
	\frac{d}{dt}V_\text{NLS}
		= 4 \int \bvec{x} \cdot \text{Im}
			\left\{	\psi^*\nabla\psi \right\}d\bvec{x} ,\qquad 
	V_\text{NLS}(t) = \int \abs{\bvec{x}}^2 \abs{\psi}^2 d\bvec{x}\,.
\]
Therefore, it follows that \[
	\frac{d^2}{dt^2} V_\text{NLS} = 8H_\text{NLS}.
\]
In the supercritical NLS, the second derivative of the variance is not related
to this conservation law.
Nevertheless, direct differentiation shows that 
\[ \begin{gathered}
	\frac{d}{dt}V_\text{NLS}  
	= 4 \int \bvec{x} \cdot \text{Im}
			\left\{	\psi^*\nabla\psi \right\}d\bvec{x}\,, \\
	\frac{d}{dt} \int \bvec{x} \cdot \text{Im}
		\left\{	\psi^*\nabla\psi \right\}d\bvec{x} 
	= 2H_\text{NLS} - 
		\frac{\sigma d-2}{(\sigma+1)} \norm{\psi}_{2(\sigma+1)}^{2(\sigma+1)}\,.
\end{gathered}\]
Therefore
\begin{equation}	\label{eq:VID-NLS}
	\frac{d^2}{dt^2}V_\text{NLS} = 8H_\text{NLS} - 
		4\frac{\sigma d-2}{(\sigma+1)} \norm{\psi}_{2(\sigma+1)}^{2(\sigma+1)}.
\end{equation}
Since~$V_\text{NLS}\geq0$, the variance identity~\eqref{eq:VID-NLS} shows that
solutions of the critical and supercritical NLS, whose Hamiltonian is negative,
become singular in a finite time~\cite{Talanov-70}.

We next extend the analogy between~\eqref{eq:J_dilation}
and~\eqref{eq:J-dilation-NLS} to the non-critical case.
In the case of the BNLS, direct differentiation shows that 
\[
	\frac{d}{dt} \int \bvec{x} \cdot 
		\text{Im}\left\{\psi^*\nabla\psi\right\}d\bvec{x}
		= 4H - \frac{\sigma d-4}{2(\sigma+1)}
			\norm{\psi}_{2(\sigma+1)}^{2(\sigma+1)}.
\]
Therefore, if we define 
\begin{equation}	\label{eq:V_BNLS}
	V_\text{BNLS}(t) =
		V_\text{BNLS}(0) 
		+ \int_{s=0}^{t}\left( 
			\int \bvec{x} \cdot \text{Im}\left\{
				\psi^*(s,\bvec{x})\nabla\psi
			\right\}d\bvec{x}
		\right)ds, 
\end{equation}
where~$V_\text{BNLS}(0)~$ is a positive constant, we get the BNLS variance identity \[
	\frac{d^2}{dt^2}V_\text{BNLS}
		= 4H - \frac{\sigma d-4}{2(\sigma+1)}
			\norm{\psi}_{2(\sigma+1)}^{2(\sigma+1)}.
\]
In order to use this identity to prove singularity formation, however, one must
show that~$V_\text{BNLS}$, as defined by~\eqref{eq:V_BNLS}, has to remain positive.
Direct integration by parts gives that
\[
		\int \bvec{x} \cdot \text{Im}\left\{\psi^*\nabla\psi\right\}d\bvec{x}
		= \frac1{4(d+2)} 
			\int |\bvec{x}|^4 \cdot \text{Im} \left\{
				\nabla\psi^* \Delta\nabla \psi 
			\right\}d\bvec{x} 
			+ \frac1{16(d+2)} \left( 
				\int |\bvec{x}|^4 |\psi|^2 d\bvec{x}
			\right)_t.
\]
While the second term on the RHS is a temporal derivative of a positive-definite
quantity, the first term on the RHS is not.
Therefore, it remains an open question whether~$V_\text{BNLS}$ has to be
positive.

\section{\label{sec:stationary}Numerical calculation of standing waves}

The BNLS equation~\eqref{eq:BNLS} admits the standing-wave solutions ~$$
\psi(t,\bvec{x})=\lambda^{2/\sigma}e^{i\lambda^4t} R(\lambda\bvec{x}).
$$
The equation for the standing-wave profile is
\begin{equation}	\label{eq:StandBNLS}
	-\Delta^2 R(\bvec{x})-R+|R|^{2\sigma}R=0,
\end{equation}
where~$R\in H^2$.
For example, in one dimension, equation~\eqref{eq:StandBNLS} is given by
\begin{equation}	\label{eq:1DStandBNLS}
 -R_{xxxx}(x)-R+|R|^{2\sigma}R=0,
\end{equation}
and in two dimensions by 
\begin{equation}	\label{eq:2DStandBNLS}
 -\left(  R_{xxxx}(x,y)+2R_{xxyy}+R_{yyyy}\right)-R+|R|^{2\sigma}R=0.
\end{equation}
\begin{subequations}	\label{eqs:radial_stationary_state}
	If we impose radial symmetry, eq.~\eqref{eq:StandBNLS} reduces to 
	\begin{equation}	\label{eq:radial_stationary_stateBNLS}
		-\Delta^2_rR(r)-R+|R|^{2\sigma}R = 0,
	\end{equation}
	where~$\Delta_r^2$ is given in~\eqref{eq:radial_bi_Laplacian}.
	At~$r=0$, all the odd derivatives vanish, and so the solution
	of~\eqref{eq:radial_bi_Laplacian} is subject to the boundary conditions 
	\begin{equation}	\label{eq:radial_stationary_stateBCs}
			R^\prime(0)=R^{\prime\prime\prime}(0)=R(\infty)=R^\prime(\infty)=0.
	\end{equation}
\end{subequations}

The solution of eq.~\eqref{eqs:radial_stationary_state} was computed numerically in the
one-dimensional case in~\cite{Fibich_Ilan_George_BNLS:2002} as follows.
In the~$1D$ case, eq.~\eqref{eqs:radial_stationary_state} can be integrated
once, yielding an explicit relation between~$R(0)$ and~$R^{\prime\prime}(0)$.
This parameter-reduction enables the usage of a one-parameter shooting approach.
Unfortunately, such a parameter reduction is not possible in higher dimensions.
Therefore, in multi-dimensions we compute the ground-states using the Spectral
Renormalization method (SRM), which was introduced by Petviashvili
in~\cite{Petviashvili:1976}, and more recently by Albowitz and Musslimani
in~\cite{Ablowitz-Musslimani-SLSR:2005}.
See Section~\ref{ssec:SRM} for further details.

\begin{figure}
	\begin{center}
	\subfloat[$d=1~(\sigma=4)$]{\label{fig:SRM_1D}%
		\includegraphics[angle=-90,clip,width=0.35\textwidth]{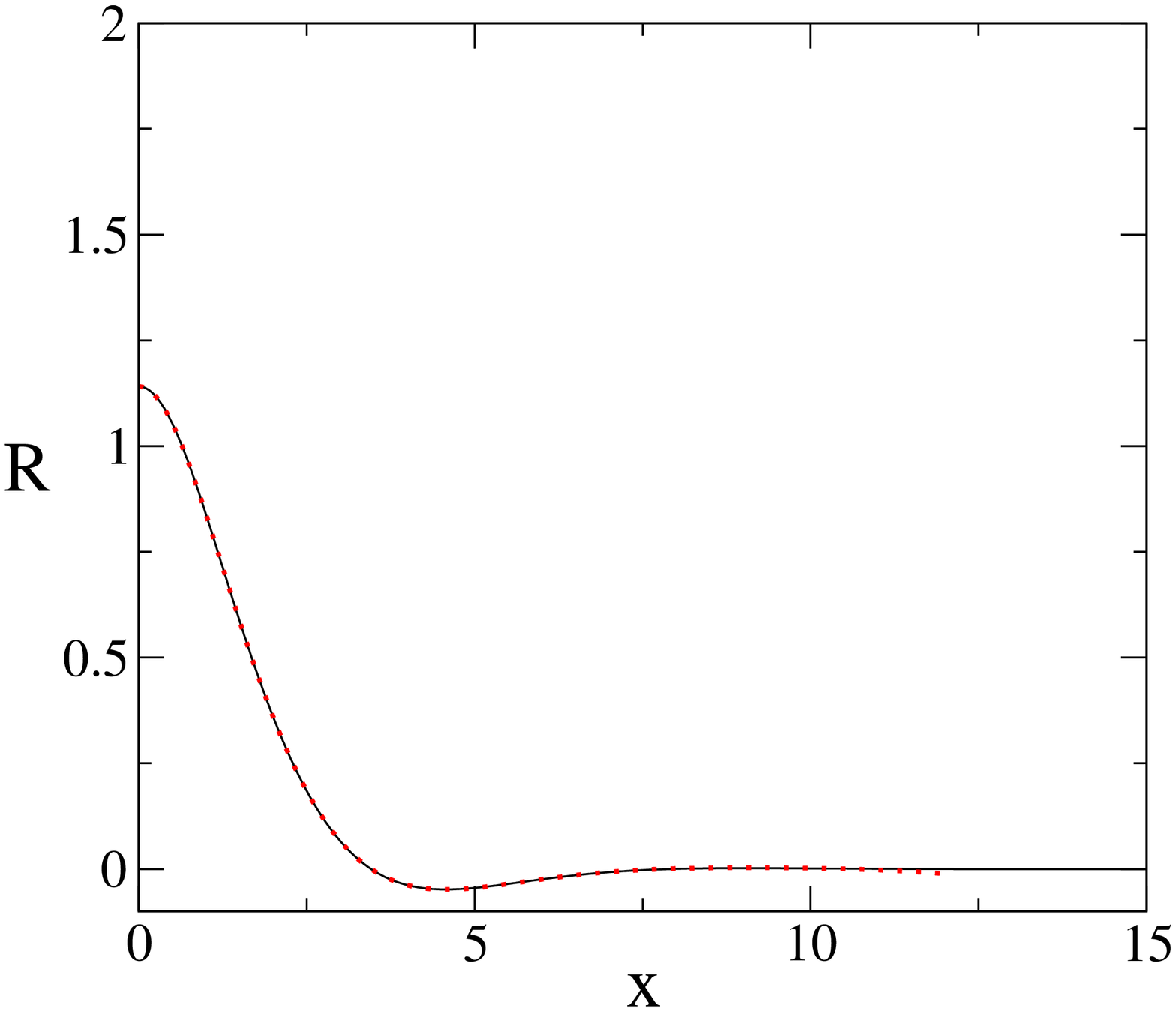} 
	}\hspace{-0.051\textwidth}
	\subfloat[$d=2~(\sigma=2)$]{\label{fig:SRM_2D}%
		\includegraphics[angle=-90,clip,width=0.35\textwidth]{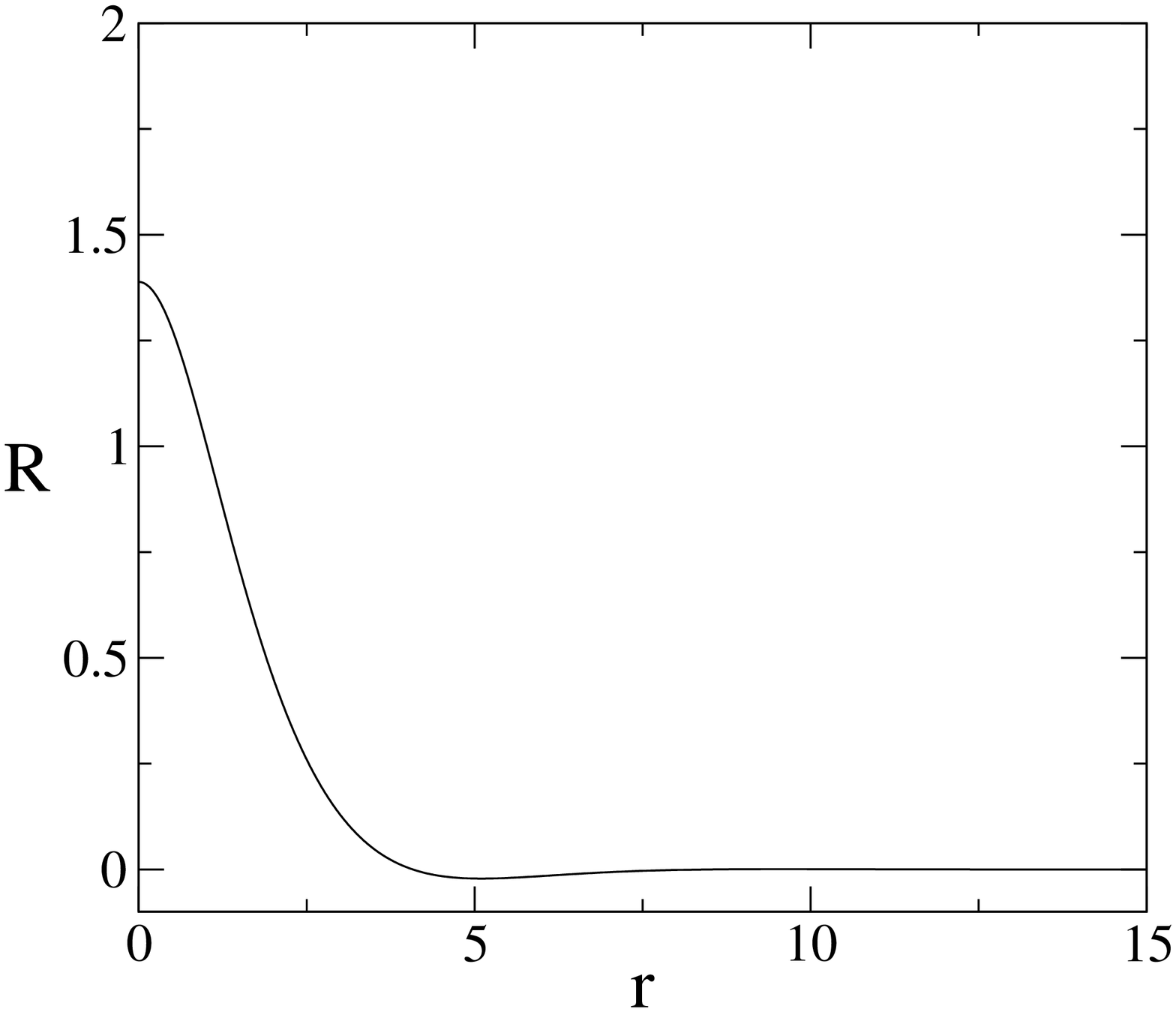} 
	}\hspace{-0.051\textwidth}
	\subfloat[$d=3~(\sigma=4/3)$]{\label{fig:SRM_3D}%
		\includegraphics[angle=-90,clip,width=0.35\textwidth]{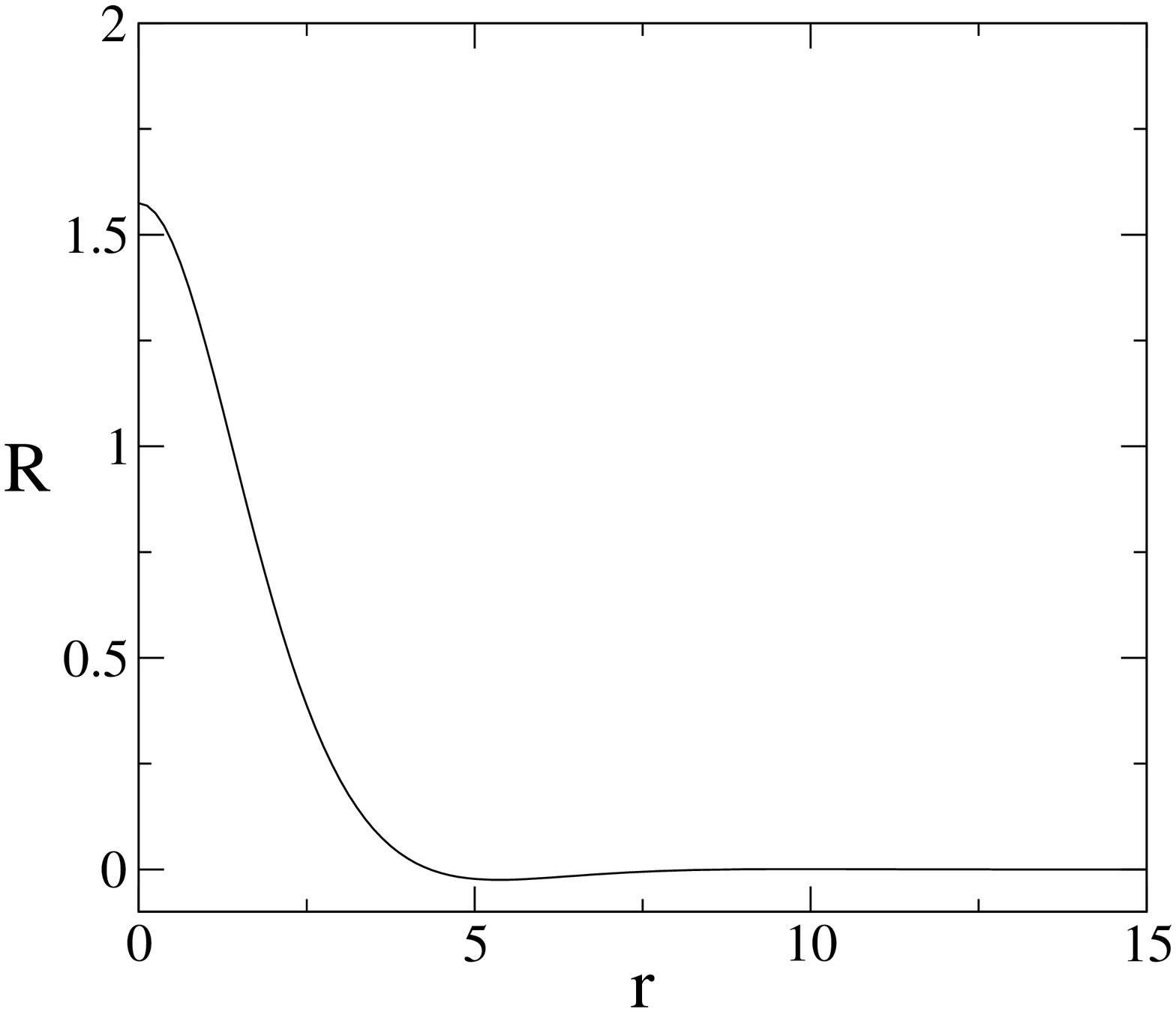} 
	}

	\mycaption{\label{fig:SRM} 
		Numerical solution of eq.~\eqref{eq:crit_stationary_state}, using the
		SRM method (solid line).
		Red dotted line in (A) is the solution computed
		in~\cite{Fibich_Ilan_George_BNLS:2002} using the shooting method. 
	}
	\end{center}
\end{figure}
In Figure~\ref{fig:SRM} we display the results in the critical case, i.e., the
ground states of 
\begin{equation}	\label{eq:crit_stationary_state}
	-\Delta^2 R -R + |R|^{8/d}R = 0\,.
\end{equation}
Figure~\ref{fig:SRM_1D} displays the ground-state of eq.~\eqref{eq:StandBNLS} in
the critical~$1D$ case, as calculated by the SRM.
The solution is in excellent agreement with the solution computed 
in~\cite{Fibich_Ilan_George_BNLS:2002} using the shooting method.
Figure~\ref{fig:SRM_2D} and Figure~\ref{fig:SRM_3D} display the ground state in the
critical~$2D$ and~$3D$ cases.
We note that, while the SRM method that we use does not enforce radial symmetry, the
calculated ground states for~$d=2$ and~$d=3$ are radially symmetric (data not shown).
As noted in~\cite{Fibich_Ilan_George_BNLS:2002}, the ground-states of the BNLS
are non-monotonic in~$r$ and change their sign, in contradistinction with the
ground-states of the NLS which are monotonically-decreasing and strictly
positive.

\section{\label{sec:critical_power}Critical power for collapse }

Theorem~\ref{thrm:GE_critical} shows that the critical power for collapse in the
critical BNLS~\eqref{eq:CBNLS} is~$\BPCrit=\norm{R}_2^2$, when~$R$ is the
ground-state of~\eqref{eq:crit_stationary_state}.
The computation of~$R$, see Section~\ref{sec:stationary}, allows for the
numerical calculation of the critical power~$\BPCrit$.
The case~$d=1$ was found in~\cite{Fibich_Ilan_George_BNLS:2002} to be \[
	\BPCrit(d=1)=\int_{x=-\infty}^{\infty}\abs{R(x)}^2dx \approx 2.9868\,.
\]
Using the calculated ground state in the two-dimensional case, see
Figure~\ref{fig:SRM_2D}, we now calculate the critical power in the
two-dimensional case, giving \[
	\BPCrit(d=2) = \iint_{x,y=-\infty}^{\infty}|R(x,y)|^2dxdy \approx 13.143\,.
\]
Similarly, using the calculated ground state in the~$3D$ case, see
Figure~\ref{fig:SRM_3D}, gives 
\[
	\BPCrit(d=3) = \iiint_{x,y,z=-\infty}^{\infty}|R(x,y,z)|^2dxdydz \approx 44.88\,.
\]

\begin{figure}
\centering
	\subfloat[$d=1, \sigma=4$]{%
		\includegraphics[angle=-90,clip,width=0.4\textwidth]%
			{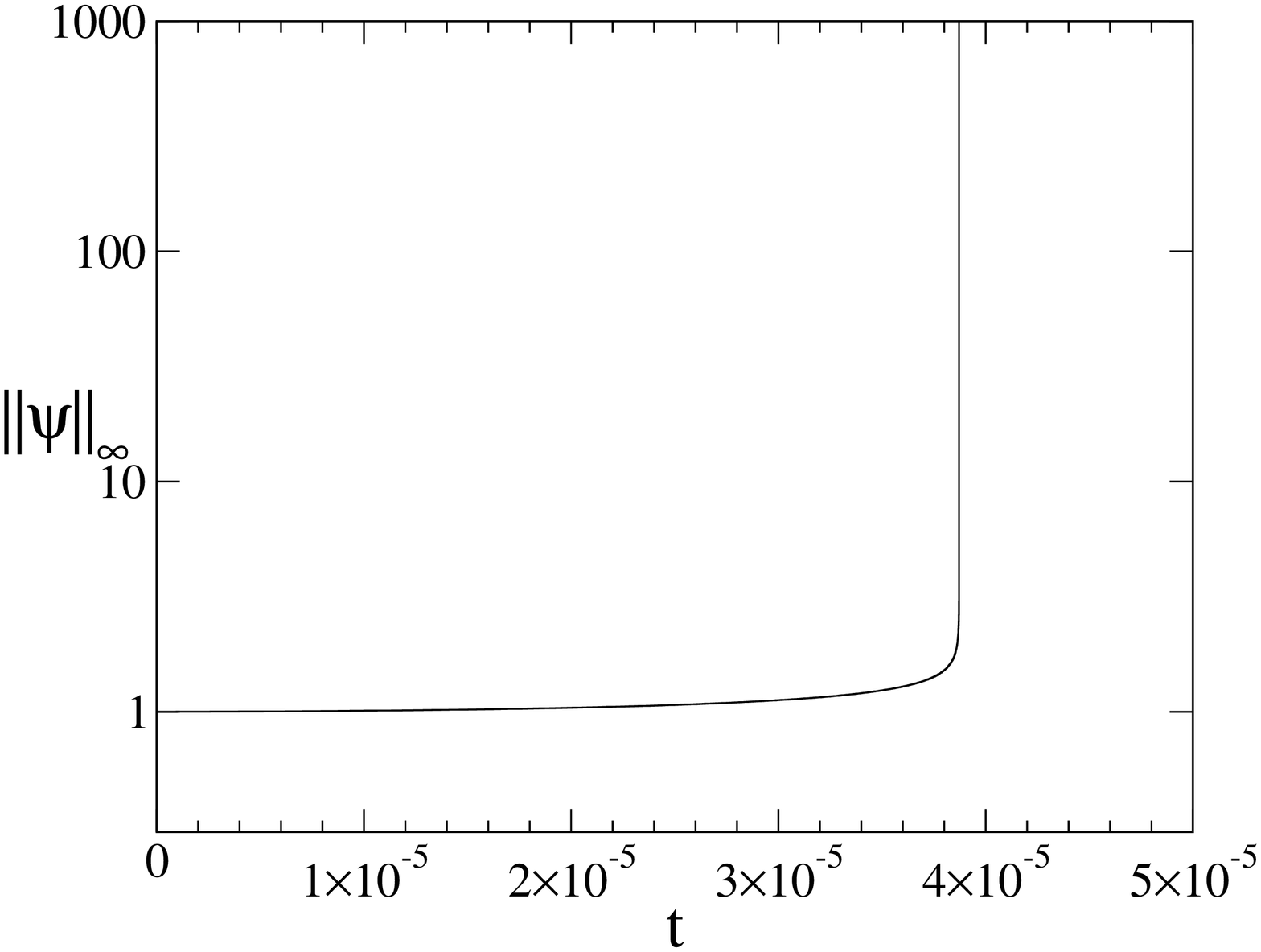} 
	}
	\subfloat[$d=2, \sigma=2$]{%
		\includegraphics[angle=-90,clip,width=0.4\textwidth]%
			{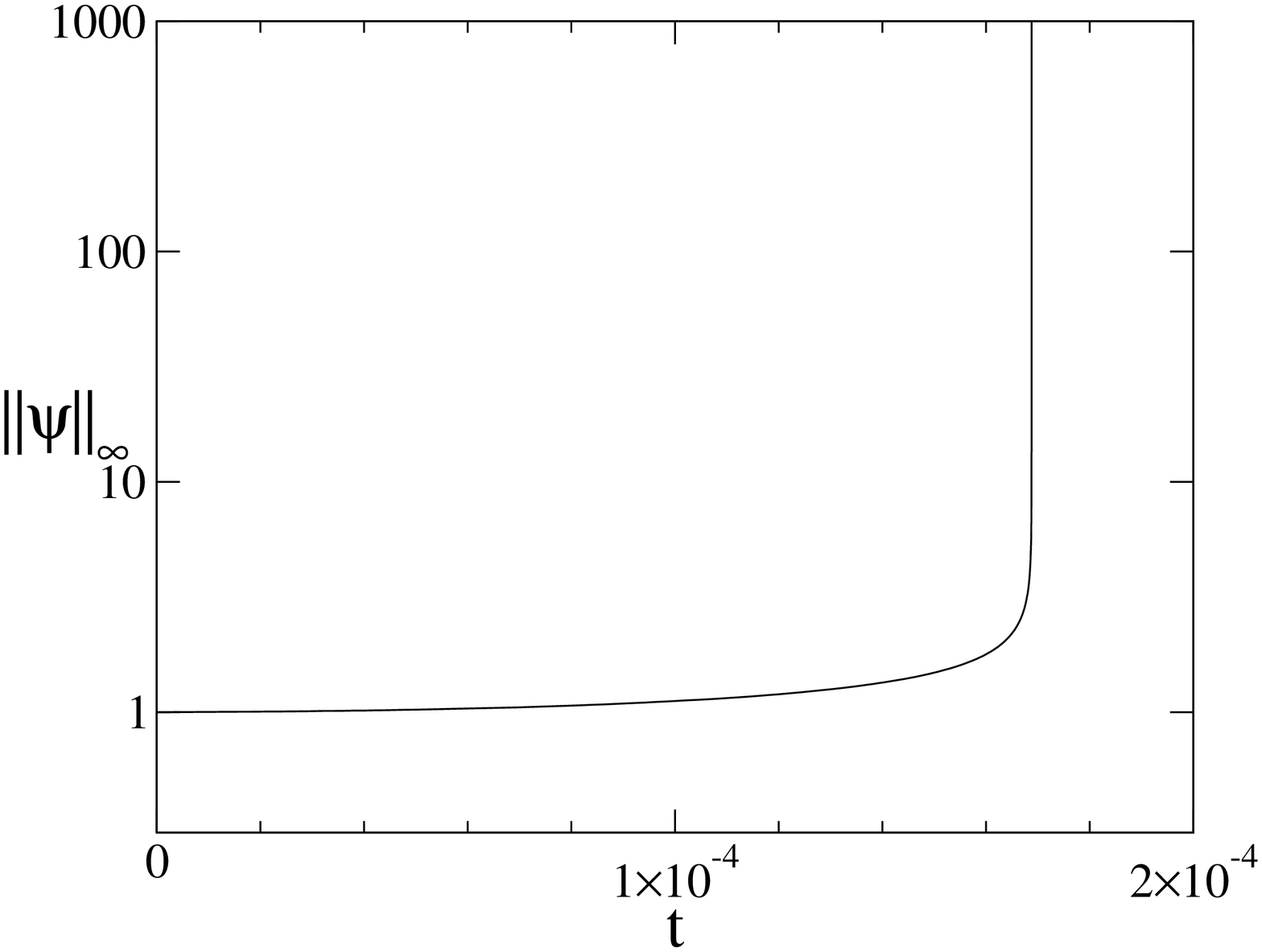} 
	}

	\mycaption{\label{fig:peak_RIC}
		Collapsing solutions of the critical BNLS~\eqref{eq:radial_CBNLS} with the
		perturbed ground-state initial condition 
		$\psi_0(\bvec{x})=1.001\cdot R(\bvec{x})$.
	}
\end{figure}

We now ask whether Theorem~\ref{thrm:GE_critical} is sharp, in the sense that
for any~$\varepsilon>0$, there exists an initial condition~$\psi_0\in H^2$ such
that~$\|\psi_0\|_2^2\le(1+\varepsilon)\BPCrit$ and the corresponding solution of
the critical BNLS becomes singular.
As noted, at present there is no proof that solutions of the BNLS can become
singular.
Therefore, in particular, it is unknown whether Theorem~\ref{thrm:GE_critical}
is indeed sharp.
Hence, we will explore this issue numerically.

We recall that in the critical NLS, the necessary condition for collapse
$\|\psi\|_2^2\ge\PCrit^\text{NLS}:=\|R^\text{NLS}\|_2^2$  is sharp, since 
for any~$\varepsilon>0$ the initial
condition~$\psi_0=(1+\varepsilon)R^\text{NLS}$ becomes singular in a finite
time~\cite{Weinstein-83}.
Therefore, we now check numerically whether for~$0<\varepsilon\ll1$, the
solution of the critical BNLS with the initial condition 
$\psi_0=\left.(1+\varepsilon)R(\bvec{x})\right.$ becomes singular.
To do this, we solve the one-dimensional and two-dimensional critical BNLS
equations with the initial condition~$\psi_0=\left.1.001\cdot R(\bvec{x})\right.$,
see Fig~\ref{fig:peak_RIC}.
In both cases, the solution appears to blow up, suggesting that
Theorem~\ref{thrm:GE_critical} is sharp.
See Section~\ref{ssec:SGR} for the numerical methodology used for solving the
BNLS.

We also note that in the critical NLS, if the initial condition is different from
the ground state, the power input required for collapse is strictly larger
than~$\PCrit^\text{NLS}$, see~\cite{critical-00}.
This the case also in the BNLS.
For example, in the one-dimensional critical BNLS with a Gaussian shaped initial
condition~$\psi_0=C\cdot e^{-r^2}$, the input power required for collapse is
strictly larger than~$1.003\cdot\PCrit$, see~\cite{Fibich_Ilan_George_BNLS:2002}.
In the two-dimensional critical BNLS with a Gaussian shaped initial condition,
the input power required is strictly larger than~$1.001\cdot\PCrit$,
see~\cite{Elad_thesis}.

\section{\label{sec:rigorous}Blowup rate, blowup profile, and power concentration
(critical case)}

\subsection{\label{ssec:low-bound}Lower-bound for the blowup rate}

In~\cite{Cazenave-88}, Cazanave and Weissler proved that the blowup rate for
singular solutions of the critical NLS~\eqref{eq:CNLS} is not slower than a
square-root, i.e., that~$\norm{\nabla\psi}_2\geq K(\TCrit-t)^{-1/2}$.
The analogous result for the critical BNLS is as follows:

\begin{thrm}	\label{thrm:low-bound} 
Let~$\psi$ be a solution of the critical BNLS~(\ref{eq:CBNLS}) that becomes 
singular at~$t=\TCrit<\infty$, and let~$
	\LN(t) = \norm{\Delta\psi}_2^{-1/2}  
$.
Then,~$\exists K=K(\norm{\psi_0}_2)>0$ such that  \[
	\LN(t) \leq  K(\TCrit-t)^{1/4},
	\qquad 0\leq t\leq\TCrit.
\] 

\begin{proof}
We follow the proof given by Merle~\cite{Merle-96} for the critical NLS.
For a fixed~$t$,~$0\leq t<\TCrit$, let us define
\[
	\psi_1(s,\bvec{x}) =
		\LN^{d/2}\psi(t+s\cdot \LN^4,\,\bvec{x}\cdot \LN).
\]
Then,~$\psi_1$ is defined for~$
	t+\LN^4s<\TCrit\iff s<S_c=\LN^{-4}(t)\cdot(\TCrit-t)
$, and satisfies the BNLS equation  \[
	i\partial_{s} \psi_1 + \Delta^2\psi_1 + 
	\left|\psi_1\right|^{8/d}\psi_1 = 0.
\]
Since \[
	\norm{\Delta\psi_1}_2^2 = 
	\LN^4 \norm{ 
		\Delta \psi(t+s\cdot \LN^4,\,\bvec{x}\cdot \LN) 
	}_2^2,
\]
this implies that~$
	{ \lim_{s\to S_c}}\norm{\Delta\psi_1}_2^2=\infty
$, i.e., that~$\psi_1(s)$ becomes singular as~$s\to S_C$.
In addition, 
\begin{subequations}	\label{eqs:psi1_0_bounds} 
	\begin{equation}	\label{eq:Delta_psi1_bound}
		\norm{\Delta\psi_1(s=0,\bvec{x})}_2^2  =
			\LN^4\norm{\Delta\psi(t,\bvec{x})}_2^2=1.
	\end{equation}

	From the definition of~$\psi_1$ and power conservation it follows that
	\begin{equation} \label{eq:psi1_bound}
		\norm{\psi_1(s=0,\bvec{x})}_2^2 
		=\norm{\psi(t,\bvec{x})}_2^2 
		= \norm{\psi_0(\bvec{x})}_2^2.
	\end{equation}
	Using equations (\ref{eq:Delta_psi1_bound}) and (\ref{eq:psi1_bound})
	and the Cauchy-Schwartz inequality gives 
	\begin{equation}
		\norm{\nabla\psi_1(s=0,\bvec{x})}_2^2  \leq
			\norm{\psi_1(s=0,\bvec{x})}_2 \cdot 
				\norm{\Delta\psi_1(s=0,\bvec{x})}_2
			= \norm{\psi_0(\bvec{x})}_2.
	\end{equation}
\end{subequations}

Together, the three formulae~(\ref{eqs:psi1_0_bounds}) imply that for any fixed
$t\in[0,\TCrit),$ 
\begin{equation}	\label{eq:psi1_bound_combined} 
	\norm{\psi_1(s=0,\bvec{x})}_{H^2}^2 
	\leq \norm{\psi_0}_2^2 + \norm{\psi_0}_2 + 1.
\end{equation}
In other words, for each~$t$, the initial~$H^2$ norm of~$\psi_1$ is bounded by a
function of~$\norm{\psi_0}_2$.
Specifically, this bound is independent of~$t$. 
From the local existence theory~\cite{Ben-Artzi-00},~$\psi_1$ exists in 
$s\in[0,S_M(t)]$, where~$
	S_M = S_M\left(\norm{\psi_1(s=0,\bvec{x})}_{H^2}\right).
$
Therefore, it follows from~\eqref{eq:psi1_bound_combined} that~$S_M$ depends on
$\norm{\psi_0}$, but is independent of~$t$.

Since~$\psi_1$ blows up at~$S_c$ we have that \[
	S_M  \leq  S_c(t)=\LN^{-4}(t)\cdot(\TCrit-t),
\] from which the result follows.
\end{proof}
\end{thrm} 	

\subsection{\label{ssec:ss-profile}Convergence to a quasi self-similar 
	blowup profile}

In~\cite{Weinstein-89}, Weinstein showed that the collapsing core of all
singular solutions of the critical NLS approaches a self-similar profile.
We now prove the analogous result for the critical BNLS:

\begin{thrm}	\label{thrm:self-similarity}
Let~$d\geq 2$ and let~$\psi(t,r)$ be a solution of the radially-symmetric 
critical BNLS~(\ref{eq:CBNLS}) with initial conditions~$
	\psi_0(r)\in H^2_{\rm radial}
$, that becomes singular at~$t=\TCrit<\infty$.
Let~$ \LN(t) = \norm{\Delta\psi}_2^{-1/2}$ and let \[
	S(\psi)(t,r)  =  \LN^{d/2}(t)\psi(t,\, \LN(t)r).
\]
Then, for any sequence~$t'_k\to \TCrit$ there is a subsequence~$t_k$ such that 
$S(\psi)(t_k, r)\to\Psi(r)$ strongly in~$L^q$, for all~$q$ such that%
\footnote{
	In fact,~$q=2(\sigma+1)$, where~$\sigma$ is in the $H^2$-subcritical  
	regime~\eqref{eq:admissible-range}.
}%
\begin{equation}	\label{eq:thrm-self-q}
	\begin{cases}
		2<q\le\infty	&	2\leq d\leq 4, \\
		2<q<\frac{2d}{d-4} 	&	4<d .
	\end{cases}
\end{equation}%
In addition,~$\norm{\Psi}_2^2\geq \norm{R}_2^2$, where~$R$ is the ground
state of equation~\eqref{eq:stationary_state}.

\begin{proof} 
Let~$t_k\to \TCrit$ and define \[
	\phi_k(r) = S(\psi)(t_k,r)
		= \LN^{d/2}(t_k) \psi(t_k,\LN(t_k)r).
\]
From the definition of~$\phi_k$ it follows that 
\[
	\norm{\phi_k}_2^2  
		= \norm{\psi_0}_2^2, \qquad
	\norm{\Delta\phi_k}_2^2  
		= \LN^4 \Vert\Delta\psi(t_k)\Vert_2^2 = 1, \qquad
	H[\phi_k] 
		= \LN^4H[\psi(t_k)].
\]
Therefore, using Cauchy-Schwartz, \[
	\Vert \nabla\phi_k \Vert_2^2 
		\leq \norm{\phi_k}_2 \cdot 
		\norm{ \Delta\phi_k }_2
		= \norm{\psi_0}_2.
\] 
Since~$\norm{\phi_k}_{H^2}$ is bounded, it follows that there exists a
subsequence of~$\phi_k$ which converges weakly in~$H^2$ to a function
$\Psi\in H_{\rm radial}^2$. 
From the Compactness Lemma~\ref{lem:compactness}, see
Appendix~\ref{app:compactness}, it follows that~$\phi_k\to\Psi$
strongly in~$L^q$, for all~$q$ given by~\eqref{eq:thrm-self-q}.

Next, we prove that~$H[\Psi]\leq0$.
Since~$\phi_k\underset{H^2}{\weakto}\Psi$, it follows that~$
	\Delta\phi_k\underset{L^2}{\weakto}\Delta\Psi
$, and so~$
	\norm{ \Delta\Psi }_2 \leq 
		\lim_{k\to\infty} \norm{\Delta\phi_k}_2 = 1
$.
Additionally, since~$\phi_k\underset{L^q}{\to}\Psi$ for some~$q=2+2s>2$, we have
that~$
	\norm{\Psi}_2 = \lim_{k\to\infty} \norm{\phi_k}_2,
$ and so \[
	H[\Psi] \leq
		\lim_{k\to\infty} H[\phi_k]
		= \lim_{k\to\infty}\LN^4H[\psi_0] = 0.
\]

In addition, since \[
	0 = \lim_{k\to\infty} H[\phi_k]
		= \lim_{k\to\infty} \left(
			1-\frac{1}{\sigma+1}\norm{\phi_k}_{2(\sigma+1)}^{2(\sigma+1)}
		\right),
\]
it follows that~$
	{ \lim_{k\to\infty}}\norm{\phi_k}_{2(\sigma+1)}>0,
$
so~$\Psi\neq0$.
Therefore, Corollary~\ref{cor:P2H_bound}, see Appendix~\ref{app:gagli-niren},
implies that~$\norm{\Psi}_2^2\geq\norm{R}_2^2$.
\end{proof}\end{thrm}

\subsection{\label{ssec:power_conc}Power Concentration}

Solutions of critical NLS have the power concentration property, whereby
the amount of power that collapses into the singularity is at least 
$\PCrit^\text{NLS}=\norm{R^\text{NLS}}_2^2$, see~\cite{Weinstein-89,Merle-90}.
In what follows, we prove the analogous results for the critical BNLS.

\begin{cor}	\label{cor:power-concentration}
Let~$d\geq 2$, and let~$\psi(t,r)$ be a solution of the radially-symmetric critical
BNLS~(\ref{eq:radial_CBNLS}) that becomes 
singular at~$t=\TCrit<\infty$.
Then,~$\forall\epsilon>0$, \[
	\liminf_{t\to \TCrit} \norm{\psi(t,r)}_{L^2(r<\epsilon)}^2
	\geq  \BPCrit,
\] where~$
	\BPCrit  =  \norm{R}_2^2.
$

\begin{proof}
The result shall follow directly from Corollary~\ref{cor:power-conc-lower-bound}.
\end{proof}\end{cor}

The following Corollary shows that the rate of power-concentration is not slower
than the blowup rate~$\LN(t)$.
The NLS analogue is due to Tsutsumi~\cite{Tsutsumi-90} and
Weinstein~\cite{Weinstein-89}.

\begin{cor}
	\label{cor:power-conc-lower-bound}
Let~$d\geq 2$, let~$\psi(t,r)$ be a solution of the radially-symmetric 
critical BNLS~(\ref{eq:radial_CBNLS}) that becomes singular at~$t=\TCrit<\infty$,
and let~$ \LN(t) = \norm{\Delta\psi}_2^{-1/2}$.
Then,
\begin{enumerate}
	\item For any monotonically-decreasing function~$a(t):[0,\TCrit)\to\Real^+$ 
		such that \[
			\lim_{t\to \TCrit}a(t)=0,\qquad \text{and} \qquad 
			\lim_{t\to \TCrit}\LN/a=0, 
		\] we have that \[
			\liminf_{t\to \TCrit} \norm{\psi(t,r)}_{L^2(r<a(t))}^2  \geq \BPCrit.
		\]
	\item For any~$\epsilon>0$,~$\exists K>0$ such that \[
		\liminf_{t\to \TCrit} \norm{\psi(t,r)}_{L^2(r<K\LN(t))}^2  
			\geq (1-\epsilon) \BPCrit.
	\]
\end{enumerate}
\begin{proof}
	See Appendix~\ref{app:proof-power-conc-lower-bound}.
\end{proof}\end{cor}

\noindent
Since the second part of Corollary~\ref{cor:power-conc-lower-bound} is true 
for all~$\epsilon>0$ and since~$K\LN(t)\to 0$, 
Corollary~\ref{cor:power-concentration} follows.

The next Corollary shows that the power-concentration rate has a quartic-root
upper bound.
The analogue in NLS theory, which is a square-root upper bound, was proved 
in~\cite{Weinstein-89,Tsutsumi-90}.

\begin{cor}
	\label{cor:power-conc-upper-bound}
Let~$d\geq 2$ and let~$\psi(t,r)$ be a solution of the radially-symmetric 
critical BNLS~(\ref{eq:radial_CBNLS}) that becomes singular at~$t=\TCrit<\infty$.
Then,
\begin{enumerate}
	\item For any monotonically-decreasing~$a:[0,\TCrit)\to\Real^+$ 
		such that \[
			\lim_{t\to \TCrit}a(t)=0,\qquad \text{and} \qquad 
			\lim_{t\to \TCrit}\frac{(\TCrit-t)^{1/4}}{a(t)}=0, 
		\] we have that \[
			\liminf_{t\to \TCrit} \norm{\psi(t,r)}_{L^2(r<a(t))}^2  \geq \BPCrit.
		\]
	\item For any~$\epsilon>0$,~$\exists K>0$ such that \[
		\liminf_{t\to \TCrit} \norm{\psi(t,r)}^2_{L^2(r<K(\TCrit-t)^{1/4})}  
			\geq (1-\epsilon) \BPCrit.
	\]
\end{enumerate}
\begin{proof}
	Theorem~\ref{thrm:low-bound} implies that~$\LN(t)\leq K(\TCrit-t)^{1/4}$.
	Therefore, the result follows immediately from
	Corollary~\ref{cor:power-conc-lower-bound}.
\end{proof}\end{cor}

Recently, Chae, Hong and Lee~\cite{ChaeHongLee:2009} used the harmonic analysis method
of Bourgain~\cite{Bourgain-98} to prove that singular solutions of the critical BNLS
for~$d\ge2$ have the power-concentration property
\[
	\lim_{t\to\TCrit} \sup_{\bvec{x_0}\in \Real^d}
		\norm{
			\psi(t,\bvec{x})
		}_{L^2\left( 
			\abs{\bvec{x}-\bvec{x_0}} < \left( \TCrit-t \right)^{1/4} 
		\right)
		}^2
		> C,
\]
where~$C$ is a positive constant.
This result is more general than
Corollaries~\ref{cor:power-concentration},\ref{cor:power-conc-lower-bound},\ref{cor:power-conc-upper-bound}
in that it does not assume radial symmetry.
The proof given here, however, is considerably simpler.
More importantly, it shows that~$C=\PCrit$.

\section{\label{sec:crit_peak}Peak-type singular solutions of the critical BNLS}

In this, we consider radially-symmetric singular solutions that are
``{\em peak-type}'', i.e., for which~$\rmax(t)\equiv0$ for~$0\le t \le T_c$,
where~$
	\rmax(t) = \displaystyle \arg \max_r|\psi|
$ is the location of the maximal amplitude.

\subsection{The critical NLS - review}

The critical NLS admits singular solutions that collapse with the
universal~$\psi_{R^\text{NLS}}$ profile, i.e.,~$\psi\sim\psi_{R^\text{NLS}}$,
where 
\begin{equation}	\label{eq:psi_R_NLS_1}
	\psi_{R^\text{NLS}}(t,r)=\frac{1}{L^{1/\sigma}(t)}
		R^\text{NLS}(\rho)e^{i\tau+i\frac{L_t}{4L}r^2}, \qquad
	\tau=\int_0^t\frac{ds}{L^2(s)},\qquad
	\rho=\frac r{L(t)}.
\end{equation}
The self-similar profile~$R^\text{NLS}$ is the ground-state solution of \[
	-R +\Delta R +\abs{R}^{2\sigma} R = 0.
\] 

The blowup rate of~$L(t)$ is given by the 
{\em loglog law} \cite{Fraiman-85,Landman-88,LeMesurier-88,Merle-03}
\begin{equation}	\label{eq:logloglaw}
	L(t) \sim \left(
		\frac{2\pi( T_c- t)}{\log\abs{\log( T_c- t)}}
	\right)^\frac{1}{2},\qquad t\to T_c.
\end{equation}
Since the blowup rate~\eqref{eq:logloglaw} is slightly faster than a
square root
$
	\displaystyle 
		\lim_{t\to\TCrit}LL_t 
		= \lim_{t\to\TCrit}\frac12\left( L^2\right)_t = 0.
$
Therefore, the phase term~$
	\frac{L_t}{4L}r^2 = \frac{LL_t}8\rho^2
$ in~\eqref{eq:psi_R_NLS_1} vanishes as~$t\to\TCrit$.
Hence, the blowup profile reduces to 
\begin{equation}	\label{eq:psi_R_NLS_2}
	\psi_{R^\text{NLS}}(t,r)=\frac{1}{L^{1/\sigma}(t)}
		R^\text{NLS}(\rho)e^{i\tau}, \qquad
	\tau=\int_0^t\frac{ds}{L^2(s)},\qquad
	\rho=\frac r{L(t)}.
\end{equation}

\subsection{\label{ssec:crit_peak_analysis}Informal analysis}

We now look for the "corresponding" peak-type singular solutions of the critical
BNLS~\eqref{eq:CBNLS}.
Theorem~\ref{thrm:self-similarity} suggests that the collapsing core of the
singular solution approaches a self-similar form, i.e.,
%
\[
	\psi(t,r) \sim \psi_B(t,r), \qquad 
	0\le r\le \rho_c\cdot L(t),
\]
where
\begin{equation}	\label{eq:crit_peak_QSS-2} 
	\psi_B(t,r) = 
		\frac1{L^{d/2}(t)}	B(\rho)	e^{i\tau(t)},\qquad 
	\rho = \frac rL,
\end{equation}
and~$\rho_c=\mathcal{O}(1)$.
Substituting~\eqref{eq:crit_peak_QSS-2} into~\eqref{eq:radial_CBNLS} and
requiring that~$
	\left[ \psi_t \right]
		\sim \left[ \Delta\psi \right]
		\sim \left[ |\psi|^{d/2}\psi \right]
$ suggests that \[
	\tau(t) = \int_{s=0}^{t}\frac{1}{L^{4}(s)}ds.
\]

Let us consider the self-similar profile~$B(\rho)$.
In the singular region~$r = {\cal O}(L)$ we have that \[
	\Delta^2\psi \sim 
	\Delta^2\psi_B \sim
		\frac{e^{i\tau}}{L^{4+d/2}} \Delta_\rho^2 B , \qquad 
	\abs{\psi}^{8/d}\psi \sim
	\abs{\psi_B}^{8/d}\psi_B =
		\frac{e^{i\tau}}{L^{4+d/2}}|B|^{8/d}B,
\] and \[
	\psi_t \sim 
		\left( \psi_B \right)_t 
	\sim
		\frac{e^{i\tau}}{L^{4+d/2}} \left\{
			iB	- L_tL^3 \left(	\frac d2 B + \rho B_\rho \right)
		\right\}\,. 
\]
Hence,~$B(\rho)$ satisfies 
\begin{equation}	\label{eq:peak_CBNLS_ODE_1}
	-B(\rho) -\Delta^2_\rho B +\abs{B}^{8/d} =
	i \left( 
		\lim_{t\to\TCrit} L_tL^3
	\right) \left(	
		\frac d2 B + \rho B_\rho 
	\right)\,.
\end{equation}

Theorem~\ref{thrm:low-bound} shows that that blowup rate of~$L(t)$ is
lower-bounded by a quartic-root.
In the critical NLS the blowup rate of peak-type solutions is
slightly faster than the analogous square-root rate, due to the loglog
correction.
Hence, we expect that the blowup rate of peak-type critical BNLS solutions
is slightly faster than a quartic root, i.e.,
\begin{equation}	\label{eq:L_faster_than_14}
	\frac{L(t)}{\sqrt[4]{\TCrit-t}} \to 0.
\end{equation} 
In that case,~$ 
\displaystyle 
	\lim_{t\to\TCrit} L_tL^3 = 
	\lim_{t\to\TCrit} \frac14 \left( L^4 \right)_t = 0,
$
and~\eqref{eq:peak_CBNLS_ODE_1} reduces to the standing-wave
equation~\eqref{eq:crit_stationary_state}.
Since the ground-states of~\eqref{eq:crit_stationary_state} attain their maximal
amplitudes at~$\rho=0$, see Section~\ref{sec:stationary}, these are peak-type
solutions.

The above informal analysis thus leads to the following Conjecture:
\begin{conj}	\label{conj:crit_peak_rate_profile} 
	The critical BNLS admits peak-type singular solutions such that: 
	\begin{enumerate}
		\item The collapsing core approaches the self-similar profile
			\begin{subequations}	\label{eqs:crit_peak_QSS} 
				\begin{equation}
					\psi(t,r) \sim \psi_R(t,r), \qquad 
					0\le r\le \rho_c\cdot L(t), 
				\end{equation}
				where
				\begin{equation}
					\psi_R(t,r) = 
						\frac1{L^{d/2}(t)}	R(\rho)	e^{i\tau(t)},\qquad 
					\rho = \frac rL,\quad
					\tau(t) = \int_{s=0}^{t}\frac{1}{L^{4}(s)}ds,
				\end{equation}
			\end{subequations}
			and~$R$ is the ground-state of
			equation~\eqref{eq:crit_stationary_state}.
		\item The blowup rate of~$L(t)$ is slightly faster than a quartic-root,
				i.e., 
			\begin{equation}	\label{eq:rate_slightly_mad}
				\lim_{t\to\TCrit} 
					\frac{L(t)}{(\TCrit-t)^p} =
				\begin{cases}
					0	&	p=1/4 \\
					\infty	&	p>1/4
				\end{cases}\,.
			\end{equation}
	\end{enumerate}
\end{conj}

\noindent In Section~\ref{ssec:crit_peak_simulations} we provide numerical
evidence in support of Conjecture~\ref{conj:crit_peak_rate_profile}.

\subsection{\label{ssec:crit_peak_simulations}Simulations}

\begin{figure}
	\centering
	\subfloat[$d=1$]{\label{fig:peak_amp_Ls_1D}%
		\includegraphics[angle=-90,clip,width=0.35\textwidth]%
			{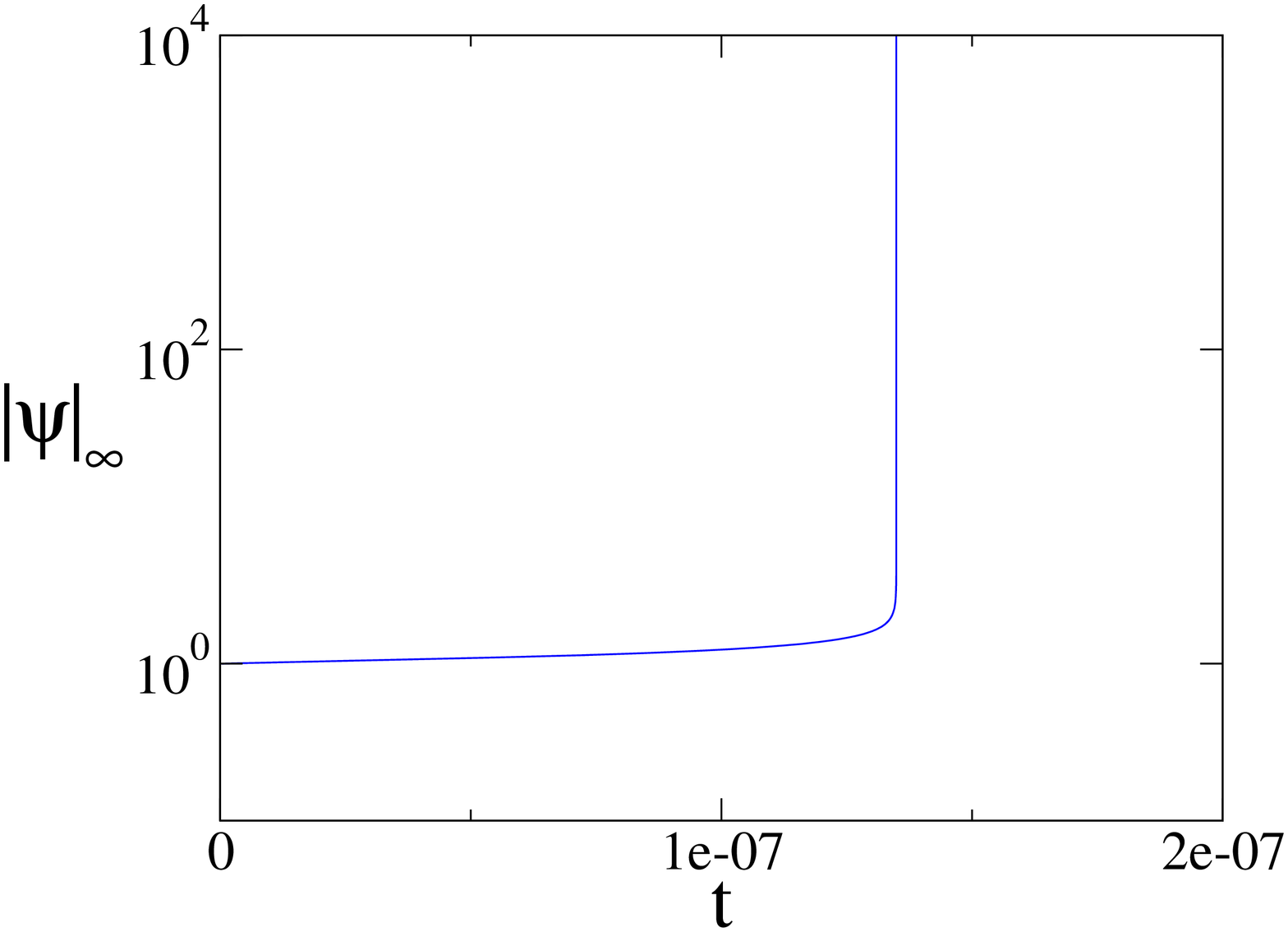}%
	}%
	\subfloat[$d=2$]{\label{fig:peak_amp_Ls_2D}%
		\includegraphics[angle=-90,clip,width=0.35\textwidth]%
			{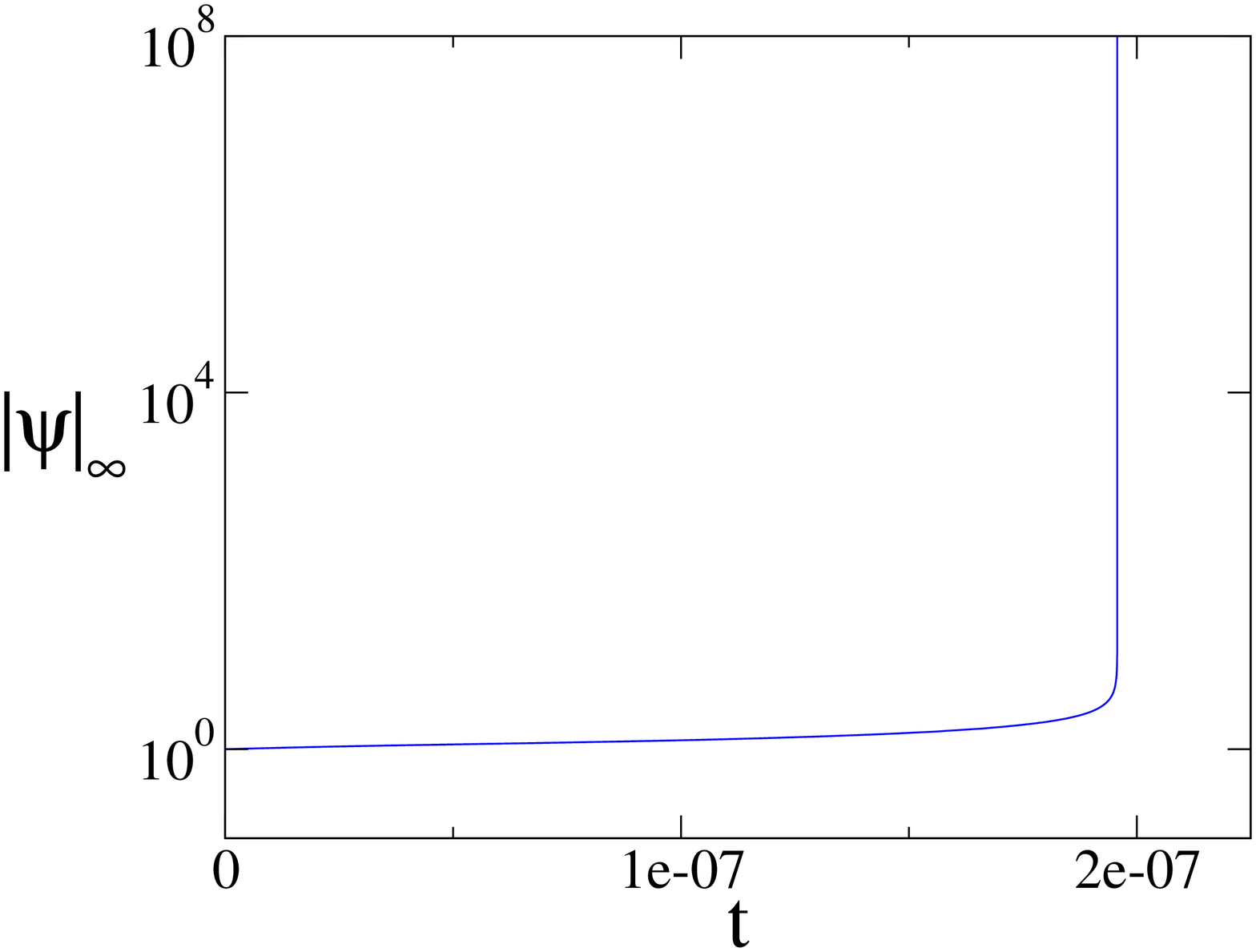}%
	}

	\mycaption{\label{fig:peak_amp_Ls}%
		Maximal amplitude of peak-type singular solutions of the critical
		BNLS~\eqref{eq:radial_CBNLS}.
	}
\end{figure}
The one-dimensional critical BNLS
\begin{equation}	\label{eq:1DCBNLS}
	i\psi_t(t,x) - \psi_{xxxx} + \left|\psi\right|^8\psi = 0
\end{equation}
was solved with the Gaussian initial condition~$
	\psi_0(x)=\left.A_1 e^{-x^2}\right.
$ with~$A_1\approx1.618$, whose power is
$\norm{\psi_0}_2^2=1.1\cdot\BPCrit(d=1)$.
The maximal amplitude of the solution~$\norm{\psi}_\infty$ as a function of time
is plotted in Fig~\ref{fig:peak_amp_Ls_1D}.
The amplitude increases abruptly by a factor of~$10^4$ around
$\TCrit\approx0.0499$, suggesting that the solution becomes singular in a finite
time.

The simulation was repeated for the radially-symmetric two-dimensional critical
BNLS
\begin{equation}	\label{eq:2DrCBNLS}
	i\psi_t(t,r) 
	- \frac{1}{r^3}\psi_r + \frac{1}{r^2}\psi_{rr} 
	- \frac{2}{r} \psi_{rrr} -\psi_{rrrr}
	+ \left|\psi\right|^{4}\psi = 0,
\end{equation}
with the Gaussian initial condition~$
	\psi_0(r)=\left.A_2 e^{-r^2}\right.
$ with~$A_2\approx3.034$, whose power is~$\|\psi_0\|_2^2=1.1\BPCrit(d=2)$
The amplitude increases abruptly by a factor of~$10^8$.
around~$\TCrit\approx 0.0606$, see Fig~\ref{fig:peak_amp_Ls_2D}, again
suggesting that the solution becomes singular in a finite time.

\begin{figure}
	\centering
	\subfloat[$d=1$]{\label{fig:crit_peak_rescaled_1D}%
		\includegraphics[angle=-90,clip,width=0.49\textwidth]%
			{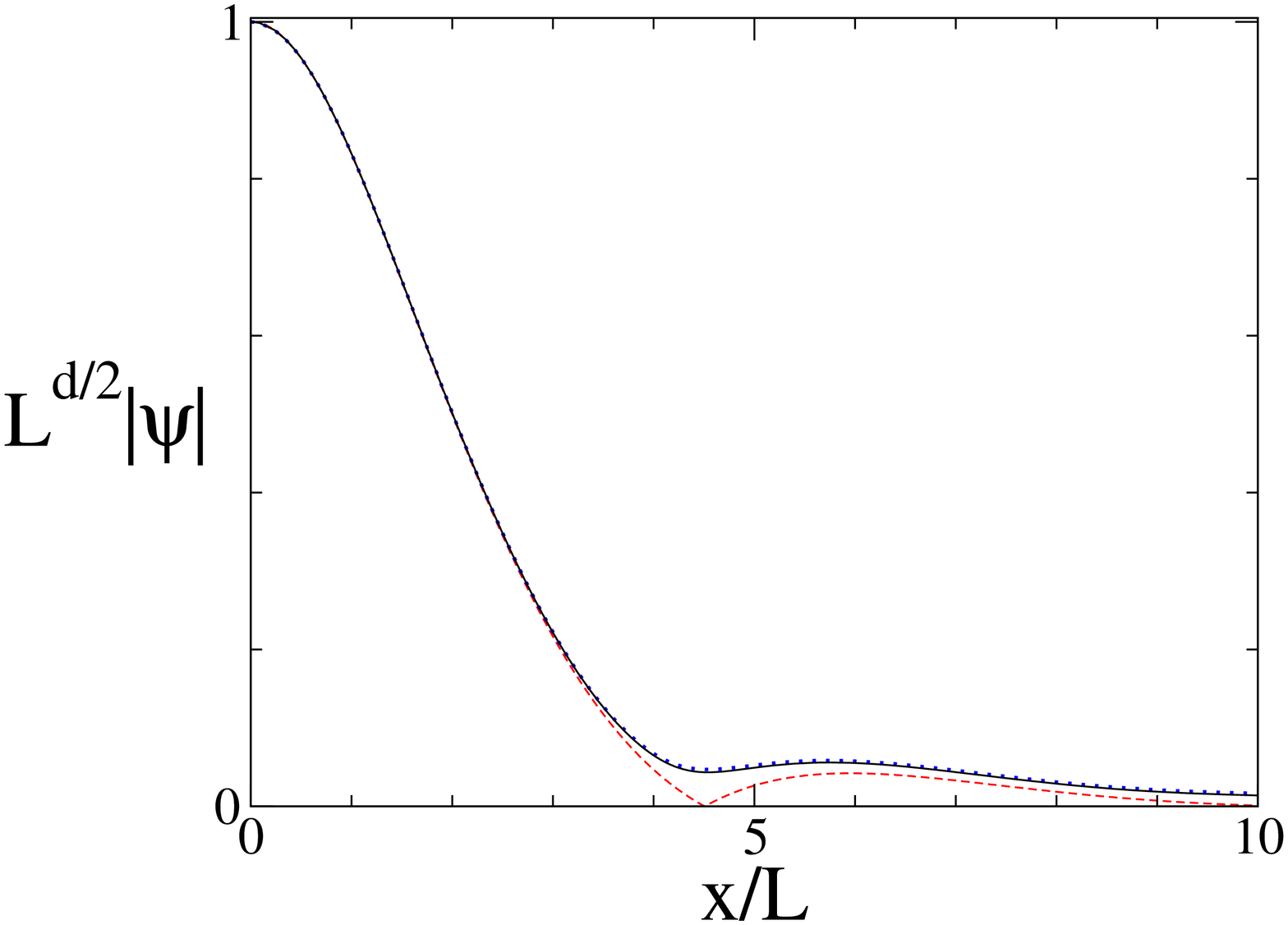}%
	}
	\subfloat[$d=2$]{\label{fig:crit_peak_rescaled_2D}%
		\includegraphics[angle=-90,clip,width=0.49\textwidth]%
			{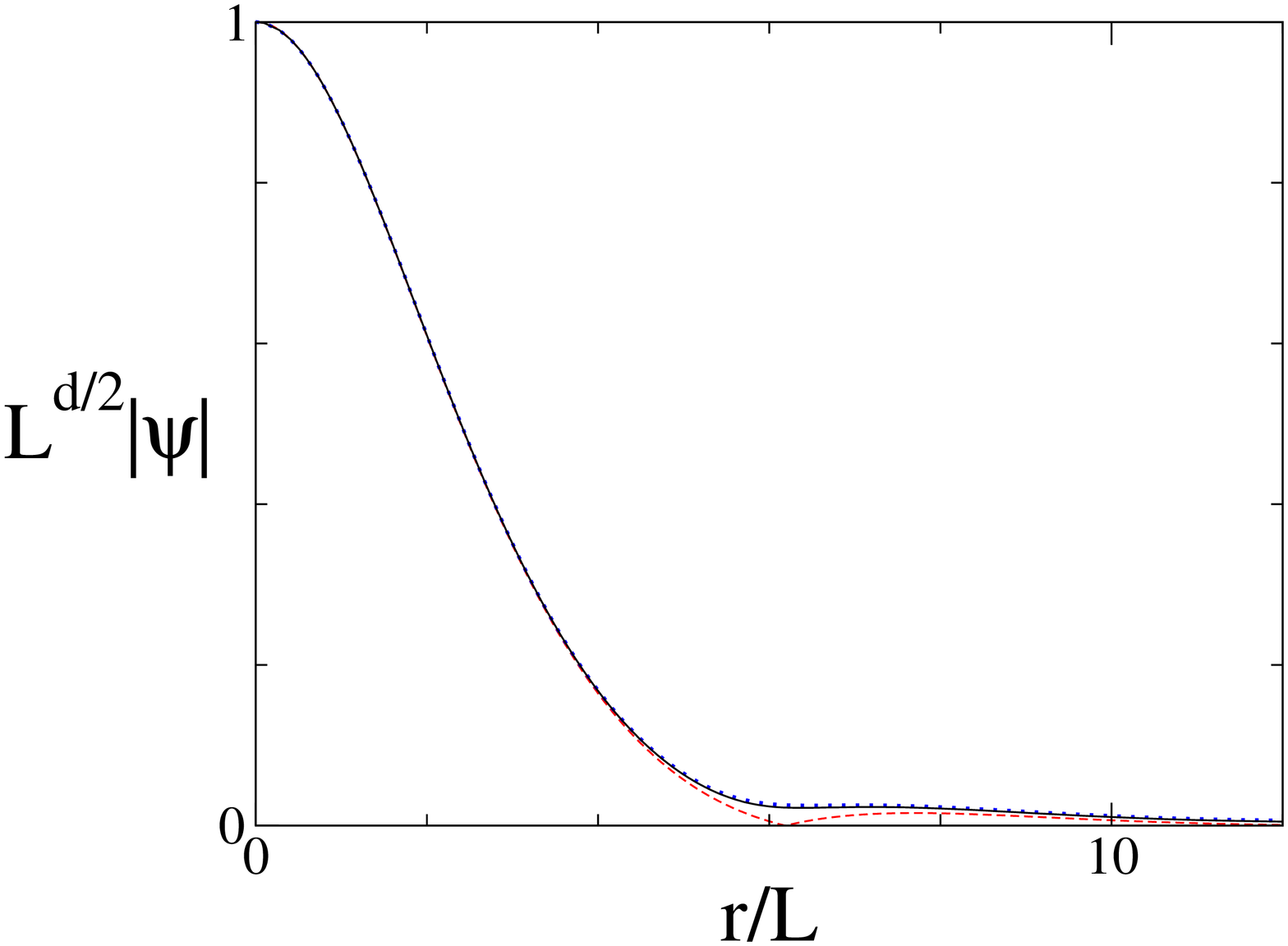}%
	}

	\mycaption{\label{fig:crit_peak_rescaled} 
	The solutions of Figure~\ref{fig:peak_amp_Ls}, rescaled according
	to~\eqref{eq:psi_rescaled_peak}, at focusing levels~$L(t)=10^{-4}$
	(blue dotted line) and~$L(t)=10^{-8}$ (black solid line). 
	Red dashed line is the rescaled ground-state~$|R|$.
	The three curves are indistinguishable for~$0\le r/L\le4$.
	}
\end{figure}
We next consider the self-similar profile of the collapsing solutions from
Figure~\ref{fig:peak_amp_Ls}.
In order to verify that it is given by~\eqref{eqs:crit_peak_QSS}, we rescale the
solutions as
\begin{equation}	\label{eq:psi_rescaled_peak}
	\psi_\text{rescaled}(t,\rho) = L^{2/\sigma}(t) \psi(t,r=\rho\cdot L), \qquad
	L(t)=\norm{\psi}_\infty^{-\sigma/2},
\end{equation}
with~$2/\sigma=d/2$.
The rescaled solutions at focusing levels of~$L=10^{-4}$ and~$L=10^{-8}$ are
indistinguishable, see Fig~\ref{fig:crit_peak_rescaled}, indicating that the
collapsing core is indeed self-similar according to~\eqref{eqs:crit_peak_QSS}.
A predicted, the self-similar profile is very close to the ground-state~$R$
in the core region~$0\le\rho\le4$.

\begin{figure}
\centering
	\subfloat[$d=1$]{%
		\includegraphics[clip,width=0.35\textwidth,angle=-90]%
			{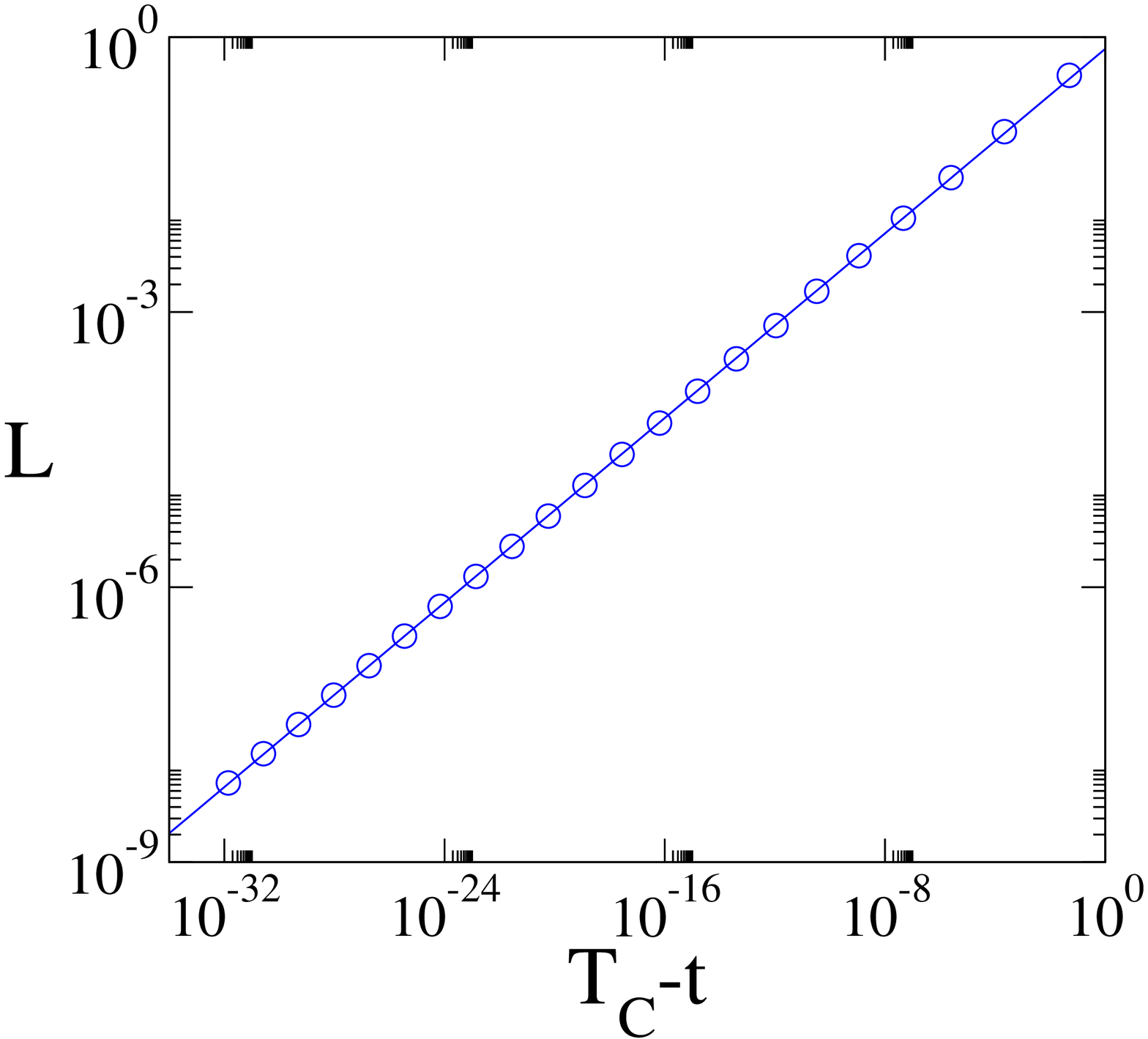}
	}
	\subfloat[$d=2$]{%
		\includegraphics[clip,width=0.35\textwidth,angle=-90]%
			{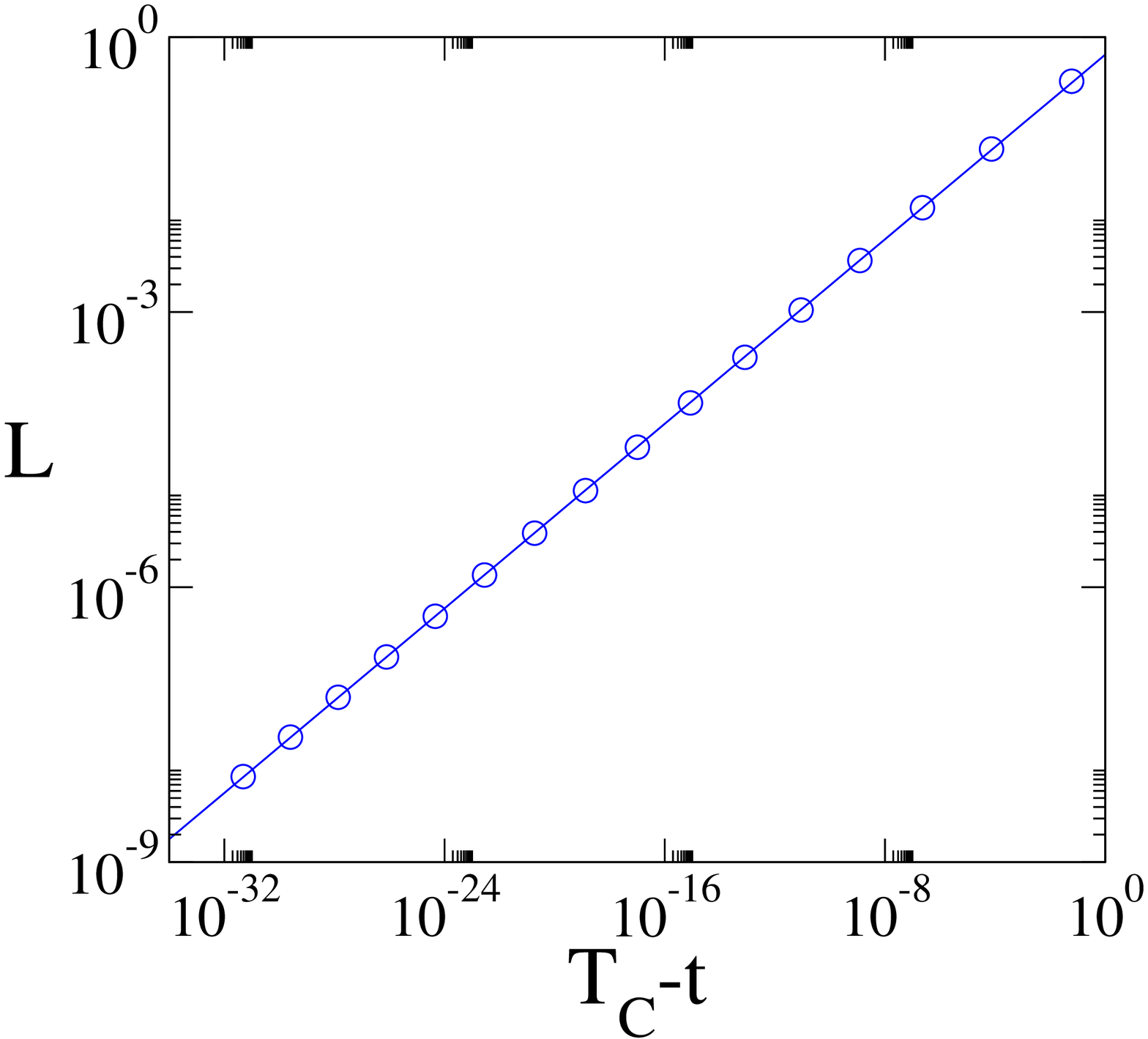}
	}

	\mycaption{\label{fig:crit_peak_powerlaw}
		$L(t)$ as a function of~$\left.(\TCrit-t)\right.$, on a logarithmic
		scale, for the solutions of Figure~\ref{fig:peak_amp_Ls} (circles).
			A)~$1D$ case. Solid line is the fitted curve
				$\left.L=0.742\cdot(\TCrit-t)^{0.2516}\right.$.
			B)~$2D$ case. Solid line is the fitted curve
				$\left.L=0.641\cdot(\TCrit-t)^{0.2516}\right.$.
	}
\end{figure}
\begin{figure}
\centering
	\subfloat[critical case]{\label{fig:crit_peak_L3Lt}%
		\includegraphics[clip,width=0.35\textwidth,angle=-90]%
			{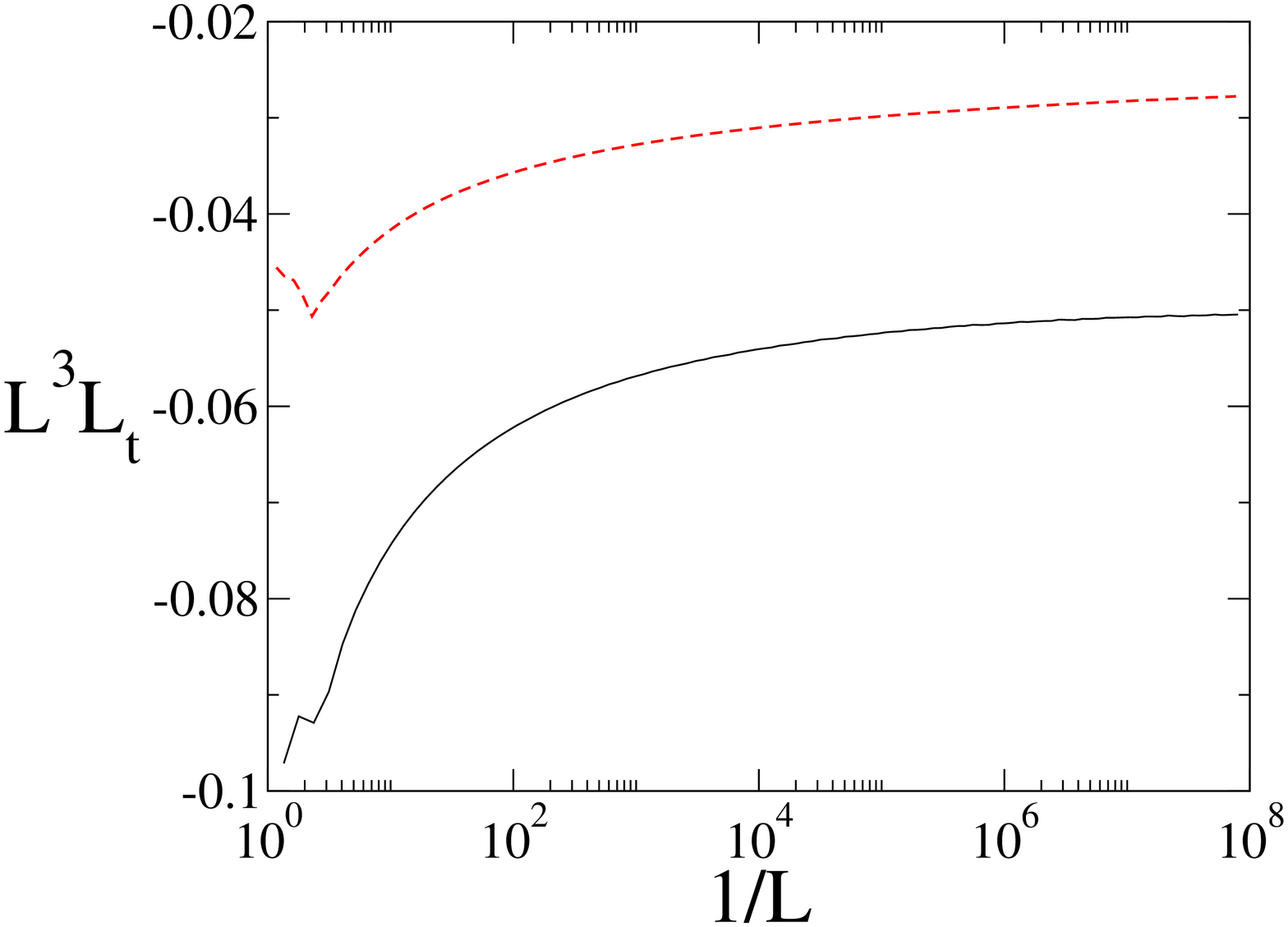}%
	}
	\subfloat[supercritical case]{\label{fig:supercrit_peak_L3Lt}%
		\includegraphics[clip,width=0.35\textwidth,angle=-90]%
			{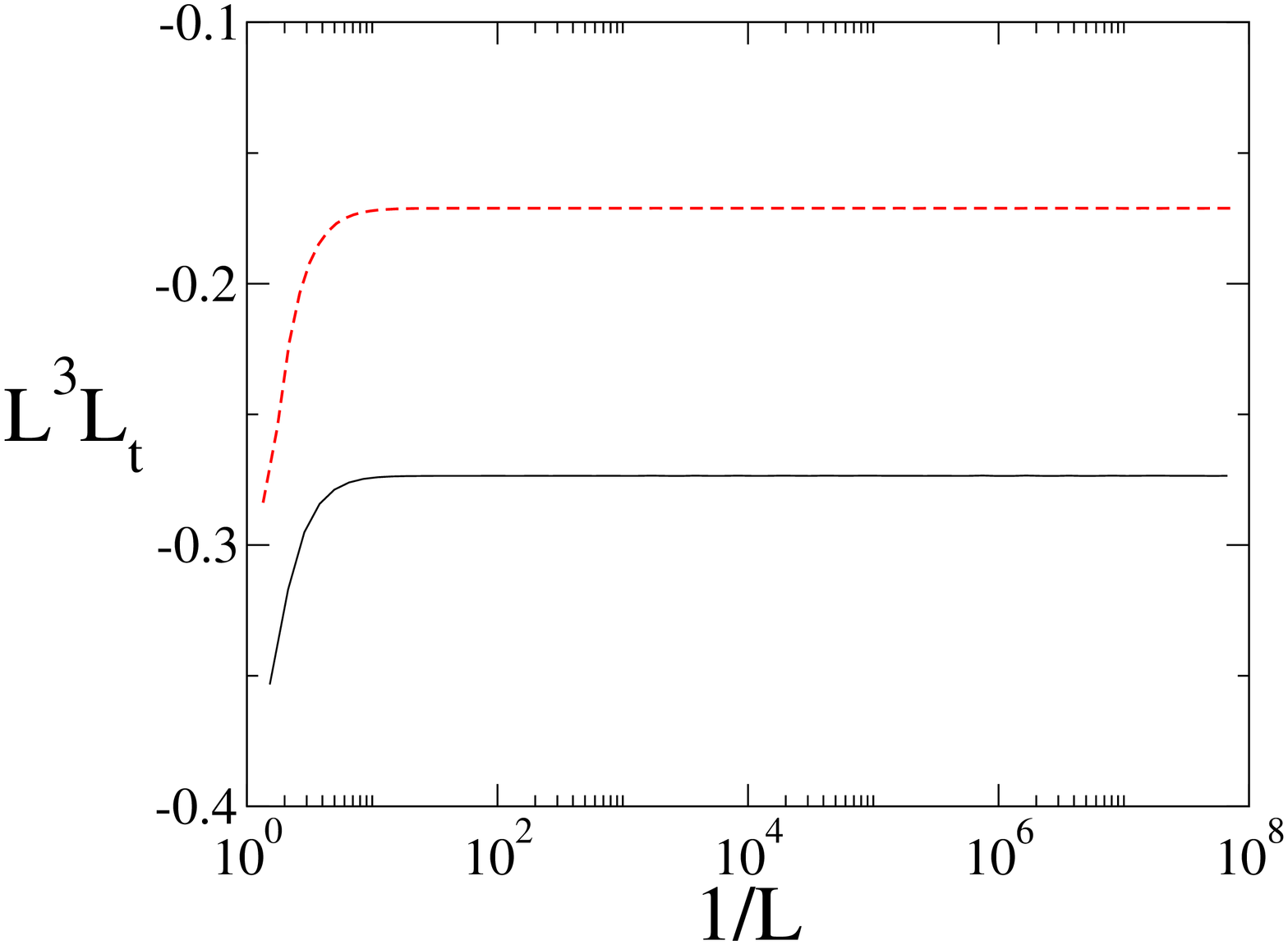}%
	}

	\mycaption{\label{peaks_L3Lt}
		A:~$L^3L_t$ as a function of~$1/L$, for the solution of
		Figure~\ref{fig:peak_amp_Ls_1D} (black solid line) and 
		of Figure~\ref{fig:peak_amp_Ls_2D} (red dashed line).
		B: same as (A) for the supercritical cases~$d=1,\sigma=6$
		(black solid line) and~$d=2,\sigma=3$ (red dashed line).
	}
\end{figure}
We next compute the blowup rate~$p$, defined by the 
relation\[L\sim \kappa (T_c-t)^p.\]
To do that, we perform a least-squares fit of~$\log(L)$ with~$\log(T_c-t)$, see
Figure~\ref{fig:crit_peak_powerlaw}, obtaining a value of~$p\approx0.2516$ for both
$d=1$ and~$d=2$.
This value of~$p$ is slightly above~$1/4$, implying that the quartic-root
lower-bound given by Theorem~\ref{thrm:low-bound} is close to the actual
blowup-rate of peak-type singular solutions.

Next, we provide two indications that the blowup rate is faster than~$1/4$.
First, if the blowup rate is exactly~$1/4$, then~$\lim_{t\to\TCrit}L^3L_t$
should be finite and strictly negative.
However, up to focusing level of~$L=10^{-8}$,~$L^3L_t$ does not appear to converge
to a negative constant, but rather to increase slowly towards~$0^-$, see
Figure~\ref{fig:crit_peak_L3Lt}.
Second, according to the informal analysis in
Section~\ref{ssec:crit_peak_analysis}, the blowup rate is faster than a quartic
root if and only if the self-similar profile~$B(\rho)$ satisfies the
standing-wave equation~\eqref{eq:stationary_state}, which is indeed what we
observed numerically in Figure~\ref{fig:crit_peak_rescaled}.

\textbf{Remark:}~
In the critical NLS the blowup rate of peak-type solutions is slightly faster
than the analogous square-root lower-bound, due to the well-known
loglog-correction~\eqref{eq:logloglaw}.
Figure~\ref{fig:crit_peak_powerlaw} shows that the blowup rate is slightly
faster than a quartic root, and Figure~\ref{fig:crit_peak_L3Lt} shows
that~$L^3L_t\to0$ very slowly. 
Together, this suggests that the blowup rate in the critical BNLS is only
slightly faster than the analogous quartic root.
At present, we do not know if the blowup rate of peak-type solutions of the
critical BNLS is a quartic root with a~$loglog$ correction.
We note, however, that the loglog correction in the critical NLS cannot be
determined numerically~\cite{PNLS-99}, and can only be derived analytically.
Therefore, we expect that the determination of the analogous correction to
the~$1/4$ blowup rate of the critical BNLS will also have to be done
analytically, and not numerically.

\section{\label{sec:supercrit_peak}Peak-type singular solutions of the supercritical BNLS}

\subsection{\label{ssec:supercrit_peak_NLS_review}The supercritical NLS - review}
In contrast to the extensive theory on singularity formation in the
critical NLS, much less is known about the supercritical NLS.
Numerical simulations and formal calculations (see,
e.g.,~\cite[Chapter 7]{Sulem-99} and the references therein)
suggest that peak-type singular solutions of the supercritical
NLS collapse with a universal~$\psi_Q$ profile, i.e.,
\begin{subequations}	\label{eq:intro_psiQ}
	\begin{equation}	\label{eq:SCNLS_quasi_parts}
		\psi(t,r) \sim
		\begin{cases}
			\psi_Q(t,r)	&	0\le r\le r_c], \\
			\psi_\text{non-singular}(t,r)\quad	& r\ge r_c ,
		\end{cases}
	\end{equation}
	where
	\begin{equation}	\label{eq:intro_psiQ_profile}
			\psi_Q(t,r)=\frac{1}{L^{1/\sigma}(t)}
				Q(\rho)e^{i\tau}, \qquad
			\tau=\int_0^t\frac{ds}{L^2(s)}, \qquad
			\rho=\frac r{L(t)}.
	\end{equation}
	Note that the singular region~$r\in[0,r_c]$ is constant in the coordinate~$r$.
	Therefore, in the rescaled variable~$\rho=r/L(t)$, the singular
	region~$\rho\in[0,r_c/L(t)]$ becomes infinite as $L(t)\to0$.
	This is in contradistinction with the critical case, wherein the singular
	region~$\rho\in[0,\rho_c]$ is constant in the rescaled variable~$\rho$, but
	shrinks to a point in the original coordinate~$r$.

	The self-similar profile~$Q$ is the solution of
	\begin{equation} \label{eq:ODE4Q}
		\begin{gathered}
			Q^{\prime\prime}(\rho) +\frac{d-1}{\rho}Q^\prime - Q 
			+i\frac {\kappa^2} 2 \left( \frac 1\sigma Q + \rho Q^\prime \right)
			+ |Q|^{2\sigma}Q = 0, \\
			\qquad Q^\prime(0)=0,\qquad Q(\infty)=0.
		\end{gathered}
	\end{equation}
\end{subequations}
Solutions of~(\ref{eq:ODE4Q}) are complex-valued, and
depend on the parameter~$\kappa$ and on the initial condition~$Q(0)=Q_0$.
Solutions of~(\ref{eq:ODE4Q}) whose amplitude~$|Q|$ is
monotonically-decreasing in~$\rho$, and which have a zero
Hamiltonian, are called {\em admissible solutions}~\cite{Sulem-99}.
For each choice of~$(\sigma,d)$, equation~\eqref{eq:ODE4Q} has a unique admissible
solution (up to a multiplication by a constant phase~$e^{i\alpha}$).
This solution is attained for specific real values of~$\kappa$ and~$Q(0)$,
which we denote as
\begin{equation}\label{eq:unique_parms_ODE4Q}
	\kappa=\kappa_Q(\sigma,d),\qquad Q(0)=Q_0(\sigma,d).
\end{equation}
The blowup rate of~$L(t)$ is a square-root, i.e.,
\begin{equation}	\label{eq:intro_psiQ_blowup_rate}
	L(t)\sim \kappa\sqrt{T_c-t},\qquad t\to T_c.
\end{equation}
Numerical simulations and formal calculations suggest that:
\begin{enumerate}
    \item The self-similar profile of singular peak-type solutions of the
		NLS~\eqref{eq:NLS} is an admissible solution of~\eqref{eq:ODE4Q}.
		Since~$Q(\rho)$ attains its maximal amplitude at~$\rho=0$, the solution
		is peak-type. 
    \item The constant~$\kappa$ of the blowup
		rate~\eqref{eq:intro_psiQ_blowup_rate} is equal
		to~$\kappa_Q(\sigma,d)$.
		Hence, in particular,~$\kappa$ is universal (i.e., is independent of the
		initial condition~$\psi_0$).
\end{enumerate}

\noindent The admissible solution~$Q(\rho)$ satisfies \[
	|Q(\rho)| \sim C\cdot\rho^{-1/\sigma}, \qquad \rho\to\infty\,.
\]
Thus,~$Q\notin L^2(\Real)$.
Nevertheless,~$Q(\rho)$ can be the self-similar profile of~$H^1$ solutions,
since~$\psi(r,t)\sim \psi_Q$ only for~$r\in [0,r_c]$, see~\cite{Berge-92}.

\subsection{\label{ssec:supercrit_peak_analysis}Informal analysis}

As in the supercritical NLS, we expect that singular peak-type solutions of the
supercritical BNLS collapse as
\begin{equation}	\label{eq:SC_quasi_parts}
	\psi(t,r) \sim
	\begin{cases}
		\psi_B(t,r)	&	0\le r\le r_c, \\
		\psi_\text{non-singular}(t,r)\quad	&	r\ge r_c ,
	\end{cases}
\end{equation}
where~$\psi_B$ is a self-similar profile, to be determined.
As in the supercritical NLS, the singular region~$r\in[0,r_c]$, is constant in
the coordinate~$r$.
Therefore, in the rescaled variable~$\rho=r/L(t)$, the singular
region~$\rho\in[0,r_c/L(t)]$ becomes infinite as~$L(t)\to0$.
This is again in contradistinction with the critical-BNLS case, where the singular
region~$\rho\in[0,\rho_c]$ is constant in the rescaled variable~$\rho$, but shrinks to
a point in the original coordinate~$r$.

The BNLS~\eqref{eq:BNLS} is invariant under the dilation symmetry~$
	r\mapsto \frac rL,
	t\mapsto \frac t{L^4},
	\psi\mapsto\frac{1}{L^{2/\sigma}}\psi
$, where~$L$ is a constant.
In the supercritical case~$\sigma d>4$, this suggests that \[
	\psi_B(t,r) = 
		\frac1{L^{2/\sigma}(t)}	B(\rho)	e^{i\tau(t)}, \qquad
		\rho=\frac rL.
\]
Similar arguments as in Section~\ref{ssec:crit_peak_analysis} show that~$
	\tau(t) = \int_{s=0}^{t}\frac{1}{L^{4}(s)}ds.
$
Therefore, as in the supercritical NLS, we expect the collapsing part of the
solution to approach the self-similar profile\footnote{
	Note that in the critical case~$2/\sigma=d/2$, hence the self-similar
	profile~\eqref{eq:supercrit_peak_QSS-0} reduces to~\eqref{eqs:crit_peak_QSS}.
}
	\begin{equation}	\label{eq:supercrit_peak_QSS-0} 
		\psi_B(t,r) = 
			\frac1{L^{2/\sigma}(t)}	B(\rho)	e^{i\tau(t)},\qquad 
		\rho = \frac rL,\qquad 
		\tau(t) = \int_{s=0}^{t}\frac{1}{L^{4}(s)}ds.
	\end{equation}

Theorem~\ref{thrm:low-bound} showed that in the critical case, if 
$L(t)\sim \kappa(\TCrit-t)^p$, then~$p\ge1/4$.
The following Lemma extends this result to peak-type solutions of the
supercritical BNLS.
\begin{lem} \label{lem:supercrit_peak_rate_profile}
	Let~$\sigma d>4$, and let~$\psi$ be a peak-type singular solution of the BNLS
	that collapses with the~$\psi_B$ profile~\eqref{eq:supercrit_peak_QSS-0}.
	If~$L(t)\sim\kappa(\TCrit-t)^p$, then~$p \ge \frac14$.
	Furthermore,~$p=1/4$ if and only if the self-similar profile~$B(\rho)$
	satisfies the equation
	\begin{equation}	\label{eq:supercrit_peak_ODE-0}
		-B(\rho) + i\frac{\kappa^4}4 \left(
				\frac{2}{\sigma}B + \rho B^\prime 
			\right)
			- \Delta_\rho B + |B|^{2\sigma}B = 0,
		\qquad	\kappa>0.
	\end{equation}
	\begin{proof}
		If~$\psi\sim\psi_B$, then \[
		\Delta^2\psi \sim \Delta\psi_B\sim
			\frac{e^{i\tau}}{L^{4+2/\sigma}} \Delta_\rho B , \qquad 
		\abs{\psi}^{2\sigma}\!\psi \sim
			\abs{\psi_B}^{2\sigma}\!\psi_B =
			\frac{e^{i\tau}}{L^{4+2/\sigma}}|B|^{2\sigma}B,
		\] and 
		\[
			\psi_t \sim \left( \psi_B \right)_t \sim
			\frac{e^{i\tau}}{L^{4+2/\sigma}} \left\{
				iB	- L_tL^3 \left(	\frac2\sigma B + \rho B_\rho \right)
			\right\}\,.
		\] 
		Hence, the equation for~$B$ is
		\begin{equation}	\label{eq:peak_SCBNLS_ODE_1}
			-B - i \left( 
				\lim_{t\to\TCrit} L_tL^3
			\right) \left(	
				\frac 2\sigma B + \rho B^\prime 
			\right)
			-\Delta^2_\rho B +\abs{B}^{2\sigma}B = 0,
		\end{equation}
		implying that~$L_t L^3$ should be bounded as~$t\to \TCrit$.
		Since~$L^3L_t \sim - p\kappa^4(\TCrit-t)^{4p-1}$, it follows
		that~$p\ge\frac14$.
		If~$p=1/4$, then~$L^3L_t \to - \frac{\kappa^4}4$, and
		equation~\eqref{eq:peak_SCBNLS_ODE_1} reduces
		to~\eqref{eq:supercrit_peak_ODE-0}.
	\end{proof}
\end{lem} 

Let us consider the fourth-order nonlinear ODE~\eqref{eq:supercrit_peak_ODE-0}
for the self-similar profile~$B(\rho)$.
Its solution requires four boundary conditions, and the determination of the
parameter~$\kappa$.
Radial symmetry implies that~$B^\prime(0)=B^{\prime\prime\prime}(0)=0$.
Since the solution is invariant up to rescaling~$
	B(\rho)\to\lambda^{2/\sigma}B(\lambda\rho)
$, one can set, with no loss of generality, $B(0)=1$.
Therefore, we require two additional constraints in order to determine~$\kappa$
and~$B^{\prime\prime}_0$.
We recall that in the supercritical NLS the ``admissible value'' of~$\kappa$ is
determined from the requirements that~$Q$ has a zero Hamiltonian and a
monotonically-decreasing amplitude, see Section~\ref{ssec:supercrit_peak_NLS_review}.
For the BNLS, the following informal argument suggests that the zero-Hamiltonian
condition should also holds.
Indeed, from Hamiltonian conservation it follows that~$
	H\left[ \psi_B \right]
$ is bounded, because otherwise the non-singular region would also have an infinite
Hamiltonian.
Calculating the Hamiltonian of~$\psi_B$, we have 
\begin{equation*}
	H\left[ \psi_B \right] \sim 
		L^{-4/\sigma+d-4} \left[ 
		\int_{\rho=0}^{r_c/L(t)} \left( 
				\abs{ \Delta_\rho B(\rho) }^2 
				-\frac{1}{1+\sigma}\abs{B}^{2+2\sigma}
			\right) \rho^{d-1} d\rho
		\right].
\end{equation*}
From $H^2$-subcriticality, see~\eqref{eq:admissible-range}, it follows
that~$L^{-4/\sigma-4+d}\to\infty$ as~$L\to0$.
Therefore, if~$H\left[ \psi_B \right]$ remains bounded as~$t\to\TCrit$ then 
\begin{equation}	\label{eq:B_Hamiltonian}
	H[B] = \int_{\rho=0}^{\infty} \left(
			\abs{ \Delta_\rho B }^2 
			-\frac1{1+\sigma} \abs{ B }^{2+2\sigma} 
		\right)
		\rho^{d-1}d\rho
	=0.
\end{equation}

WKB analysis of~\eqref{eq:supercrit_peak_ODE-0} shows that, see
Appendix~\ref{app:WKB}, \[
	B(\rho)\sim 
		c_1B_1(\rho)
		+c_2B_2(\rho)
		+c_3B_3(\rho)
		+c_4B_4(\rho)
		,\qquad
	\rho\to\infty,
\] where 
\begin{eqnarray*} 
	B_1(\rho) &\sim &
		\rho^{-\frac 2\sigma -i\frac{1}{b^3}}, \\
	B_2(\rho) &\sim&
		\frac{1}{\rho^{\frac{2}{3\sigma}(\sigma d-1)}}
		\exp\left( 
			-i\frac{3}{4}b\rho^{4/3}
			-i\frac{1}{3b^3} \log(\rho)
		\right), \\
	B_3(\rho) &\sim&
		\frac{
			\exp\left( 
			\frac{3\sqrt{3}}{8}b \rho^{4/3}
			\right)
		}{\rho^{\frac{2}{3\sigma}(\sigma d-1)}}
		\exp\left( 
			+i\frac{3}{8}b\rho^{4/3}
			-i\frac{1}{3b^3} \log(\rho)
		\right), \\
	B_4(\rho) &\sim&
		\frac{
			\exp\left( - ~
				\frac{3\sqrt{3}}{8}b \rho^{4/3}
			\right)
		}{\rho^{\frac{2}{3\sigma}(\sigma d-1)}}
		\exp\left( 
			+i\frac{3}{8}b\rho^{4/3}
			-i\frac{1}{3b^3} \log(\rho)
		\right) \\
\end{eqnarray*}
and $b=\left( \kappa^4/4 \right)^{1/3}$.
Equation~\eqref{eq:supercrit_peak_ODE-0} therefore has two
algebraically-decaying solutions,~$B_1$ and~$B_2$,
an exponentially-increasing solution~$B_3$,
and an exponentially-decreasing solution~$B_4$. 
Since~$\sigma d>4$, the exponent~$
	\frac{2}{3\sigma}(\sigma d-1)
$ of~$B_2$ is larger that the exponent~$\frac 2\sigma$ of~$B_1$, 
hence~$B_1\gg B_2$ as~$\rho\to \infty$.

\begin{lem} \label{lem:admissible_B23}
	Let $B(\rho)$ be a zero-Hamiltonian solution of~\eqref{eq:supercrit_peak_ODE-0}.
	Then, $c_2= c_3= 0$ and \[
		B(\rho) \sim c_1B_1(\rho) 
		,\qquad
		\rho\to\infty.
	\]
	Furthermore,~$B^{\prime\prime}\in L^2$.
%
\end{lem}
\begin{proof}
	The exponentially increasing solution $B_3$ must vanish identically if the
	integrals are to converge, hence~$c_3=0$.
	Next, the $\rho^{4/3}$ phase term in $B_2$ implies that \[
		\abs{B_2^{\prime\prime}}^2 \sim 
			\rho^{-\frac{4}{3\sigma}(\sigma d-1)+\frac 43}.
	\] 
	Hence, the integral \[
	\norm{B_2^{\prime\prime}}_2^2
		\sim \int \rho^{-\frac{4}{3\sigma}(\sigma d-1-\sigma) }
		\rho^{d-1}d\rho
	\]
	diverges in the $H^2$-subcritical regime~$\sigma(d-4)<4$, 
	i.e.,~$B_2^{\prime\prime} \notin L^2$. 
	Since, in addition, \[
		B_1^{\prime\prime} \in L^2, \qquad
		B_1\in L^{2+2\sigma}, \qquad
		B_2\in L^{2+2\sigma},
	\]
	the Hamiltonian can be finite only if~$c_2=0$, in which
	case~$B^{\prime\prime}\in L^2$.
\end{proof}

Lemma~\ref{lem:admissible_B23} shows that the condition~$H[B]=0$ imposes the two
constraints~$c_2=0$ and~$c_3=0$.
Hence, zero-Hamiltonian solutions of equation~\eqref{eq:supercrit_peak_ODE-0}
satisfy the five boundary conditions \[
		B(0)=1,	\quad
		B^\prime(0)=0, \quad
		B^{\prime\prime\prime}(0)=0, \qquad
		c_2\left( B_0^{\prime\prime},\kappa \right) 
		= c_3\left( B_0^{\prime\prime},\kappa \right)
		= 0.
	\]
Therefore they form a discrete set of solutions and of values of~$\kappa$.
We conjecture that the additional condition of monotonicity of~$|B|$ will lead to a
unique admissible solution~$B$ and a unique value of~$\kappa$.

\begin{cor}	\label{cor:admissible_B23_power}
	Let $B(\rho)$ be a zero-Hamiltonian solution
	of~\eqref{eq:supercrit_peak_ODE-0}.
	Then, $\norm{B}_2=\infty$.
	Nevertheless,~$
		\displaystyle \lim_{t\to\TCrit}
			\norm{\psi_B}_{L^2(r<r_c)}<\infty
	$.
\end{cor}
\begin{proof}
	Since $B(\rho)\sim c_1B_1(\rho)$, 
	\[
		\norm{B_1}_2^2 
		\sim C \int_{\rho=0}^{\infty}
			\rho^{-4/\sigma+d-1} d\rho 
			\sim C \rho^{d-4/\sigma} \Big|_{\rho=0}^{\infty}
			= \infty.
	\]
	Following the arguments of~\cite{Berge-92}, the profile~$\psi_B$ satisfies
	\begin{eqnarray*}
		\norm{\psi_B}_{L^2(r<r_c)}^2 
		&=&
		L^{d-4/\sigma}(t) \cdot 
		\int_{\rho=0}^{r_c/L(t)}
			\abs{B(\rho)}^2
			\rho^{d-1}d\rho \\
		&\sim& 
		L^{d-4/\sigma}(t) \cdot \left( 
			C \rho^{d-4/\sigma} \Big|_{\rho=0}^{r_c/L(t)}
		\right) 
		=
		\mathcal{O}(1).
	\end{eqnarray*}
\end{proof}

In summary, we conjecture the following:
\begin{conj}	\label{conj:supercrit_peak_rate_profile} ~\\
	Let $\psi$ be peak-type singular solution of the supercritical BNLS. 
	Then,
	\begin{subequations} \label{eqs:supercrit_peak_QSS}
		\begin{enumerate}
			\item The collapsing core approaches the self-similar
				profile~$\psi_B$, i.e.,
				\begin{equation}	\label{eq:supercrit_peak_QSS-1}
					\psi(t,r) \sim \psi_B(t,r), \qquad 
					0\le r\le r_c,
				\end{equation}
				where
				\begin{equation}	\label{eq:supercrit_peak_QSS-2} 
					\psi_B(t,r) = 
						\frac1{L^{2/\sigma}(t)}	B(\rho)	e^{i\tau(t)},\qquad 
					\rho = \frac rL,\qquad 
					\tau(t) = \int_{s=0}^{t}\frac{1}{L^{4}(s)}ds.
				\end{equation}
			\item The self-similar profile~$B(\rho)$ is the solution of 
				\begin{equation}	\label{eq:supercrit_peak_ODE}
					\begin{gathered}
						-B(\rho) + i\frac{\kappa^4}4 \left(
							\frac{2}{\sigma}B + \rho B^\prime 
						\right)
						- \Delta_\rho B + |B|^{2\sigma}B = 0, \\
						B(0)=1,B^\prime(0)=B^{\prime\prime\prime}(0)=0,\qquad
						H[B]=0,
					\end{gathered}
				\end{equation}
				where~$\kappa>0$ and~$H[B]$ is the Hamiltonian of~$B$,
				see~\eqref{eq:B_Hamiltonian}.
			\item In particular, $B(\rho)\neq R(\rho)$.
			\item Equation~\eqref{eq:supercrit_peak_ODE} 
				has a unique ``admissible solution'' with a unique 
				``admissible value'' of~$\kappa=\kappa(\sigma,d)$, such
				that~$|B(\rho)|$ is monotonically decreasing.
				Additionally,~$B(\rho)\sim \rho^{-2/\sigma-i4/\kappa^4}$ 
				as~$\rho\to \infty$.
			\item The admissible solution is the self-similar
				profile~$B$ of~$\psi_B$, see~\eqref{eq:supercrit_peak_QSS-2}.
			\item The blowup rate of singular peak-type solutions is exactly a
					quartic root, i.e.,\begin{equation}	\label{eq:rate_14}
					L(t)\sim\kappa\sqrt[4]{\TCrit-t},
					\qquad \kappa>0.
				\end{equation}
			\item The coefficient~$\kappa$ of the blowup rate of~$L(t)$ is equal
					to the value of~$\kappa$ of the admissible solution~$B$,
					i.e., \[
					\kappa := \lim_{t\to\TCrit} 
						\frac{L(t)}{\sqrt[4]{\TCrit-t}}  
					= \kappa(\sigma,d).
				\] 
				In particular, $\kappa$ is universal (i.e., it does not depend
				on the initial condition).
		\end{enumerate}
	\end{subequations}
\end{conj}

\noindent In Section~\ref{ssec:supercrit_peak_simulations} we provide numerical
evidence in support of Conjecture~\ref{conj:supercrit_peak_rate_profile}.

\subsection{\label{ssec:supercrit_peak_simulations}Simulations}

\begin{figure}
	\centering
	\subfloat[$d=1, \sigma=6$]{\label{fig:supercrit_peak_amp_Ls_1D}%
		\includegraphics[angle=-90,clip,width=0.35\textwidth]%
			{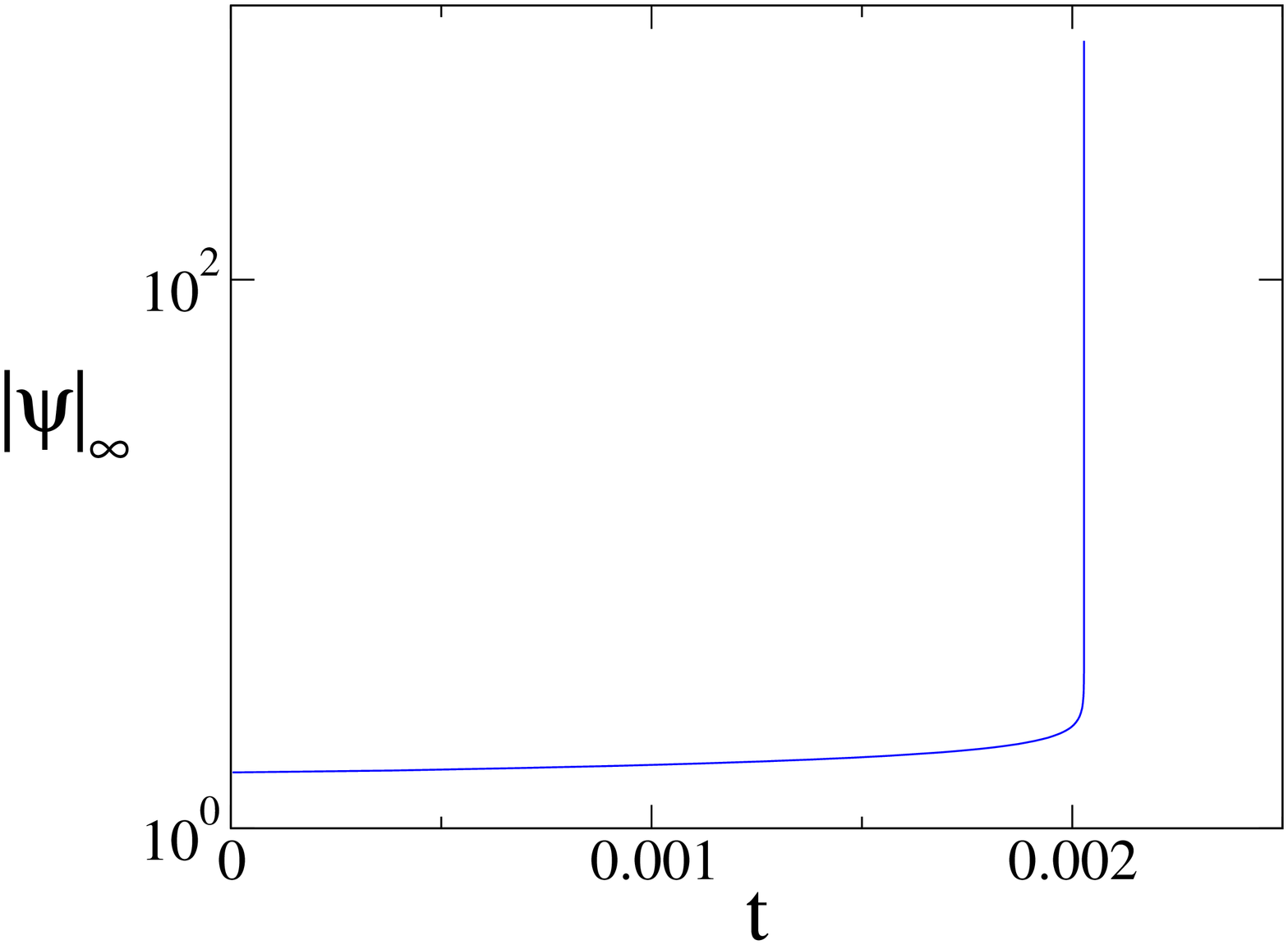}%
	}%
	\subfloat[$d=2, \sigma=3$]{\label{fig:supercrit_peak_amp_Ls_2D}%
		\includegraphics[angle=-90,clip,width=0.35\textwidth]%
			{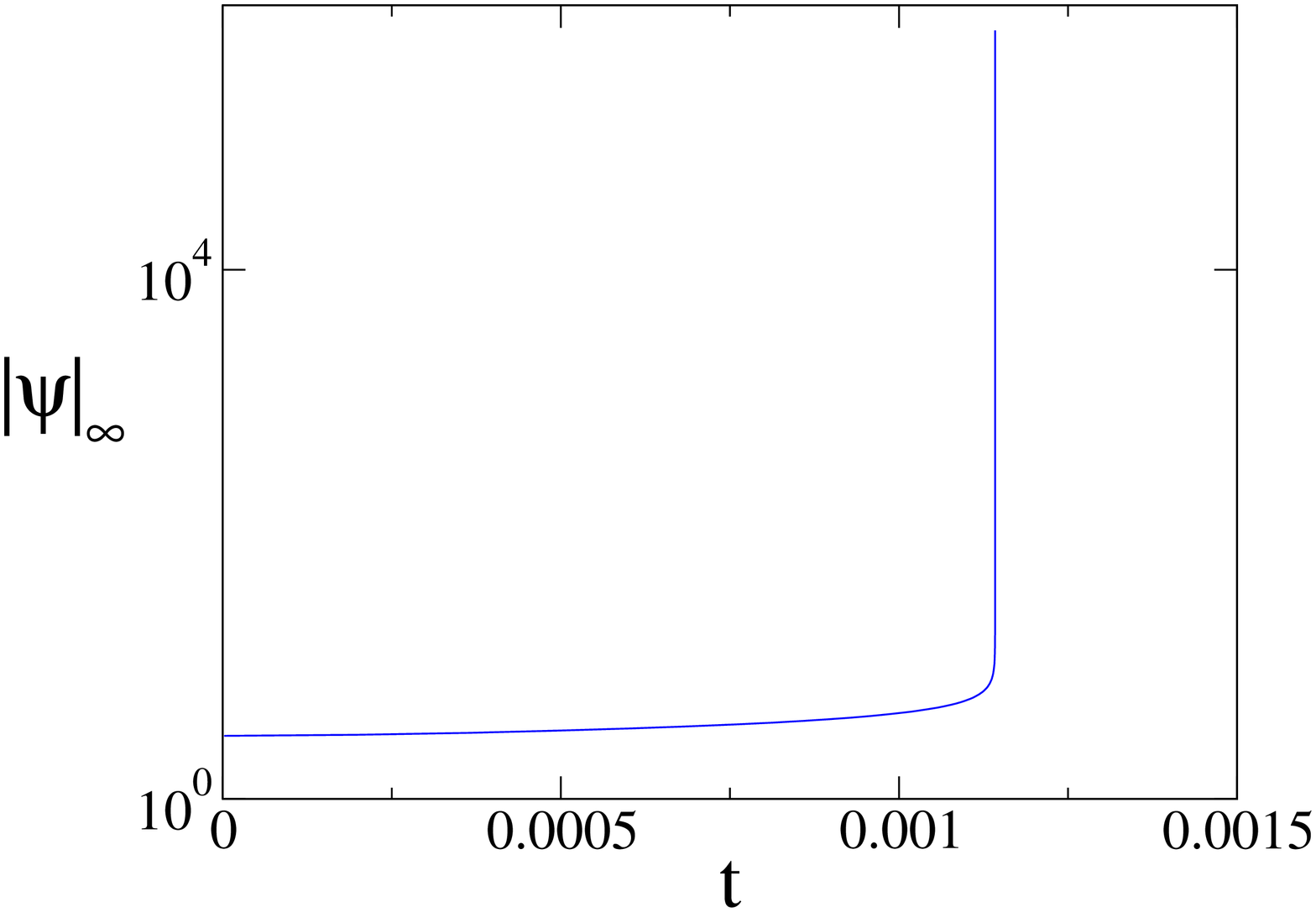}%
	}

	\mycaption{\label{fig:supercrit_peak_amp_Ls}%
		Maximal amplitude of peak-type singular solutions of the supercritical BNLS.
	}
\end{figure}

\begin{figure}
	\centering
	\subfloat[$d=1,\sigma=6$]{\label{fig:supercrit_peak_rescaled_1D}%
		\includegraphics[clip,angle=-90,width=0.45\textwidth]%
			{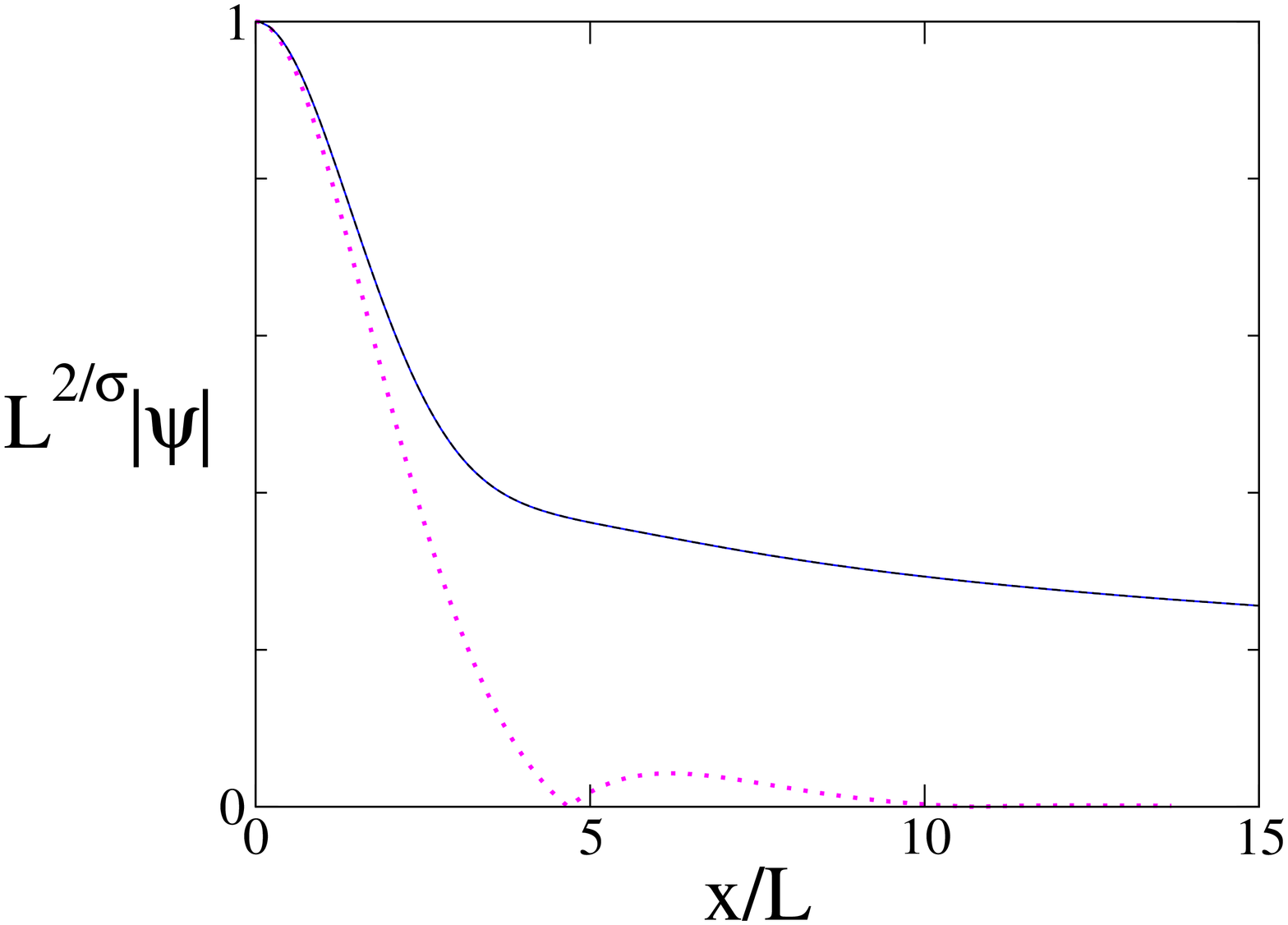}%
	}%
	\subfloat[$d=2,\sigma=3$]{\label{fig:supercrit_peak_rescaled_2D}%
		\includegraphics[clip,angle=-90,width=0.45\textwidth]%
			{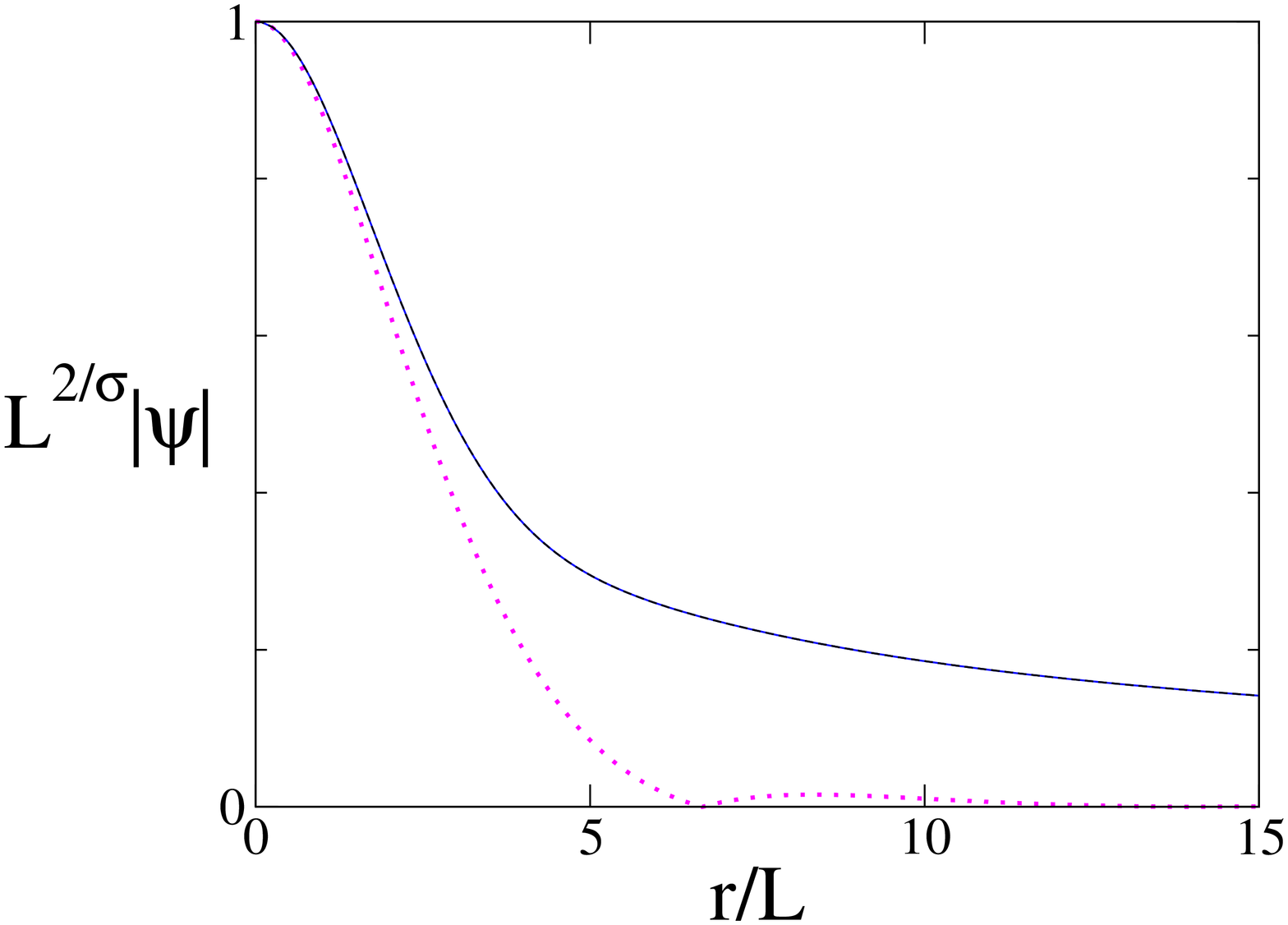}%
	}%

	\mycaption{\label{fig:supercrit_peak_rescaled}
		The solutions of Figure~\ref{fig:supercrit_peak_amp_Ls}, rescaled
		according to~\eqref{eq:psi_rescaled_peak}, at the focusing levels~$1/L=10^4$
		(blue solid line) and~$1/L=10^8$ (black dashed line).
		The magenta dotted line is the rescaled ground-state~$R$.
	}
\end{figure}

\begin{figure}
	\centering
	\subfloat[$d=1,\sigma=6$]{\label{fig:supercrit_peak_rescaled_far_1D}%
		\includegraphics[clip,angle=-90,width=0.45\textwidth]%
			{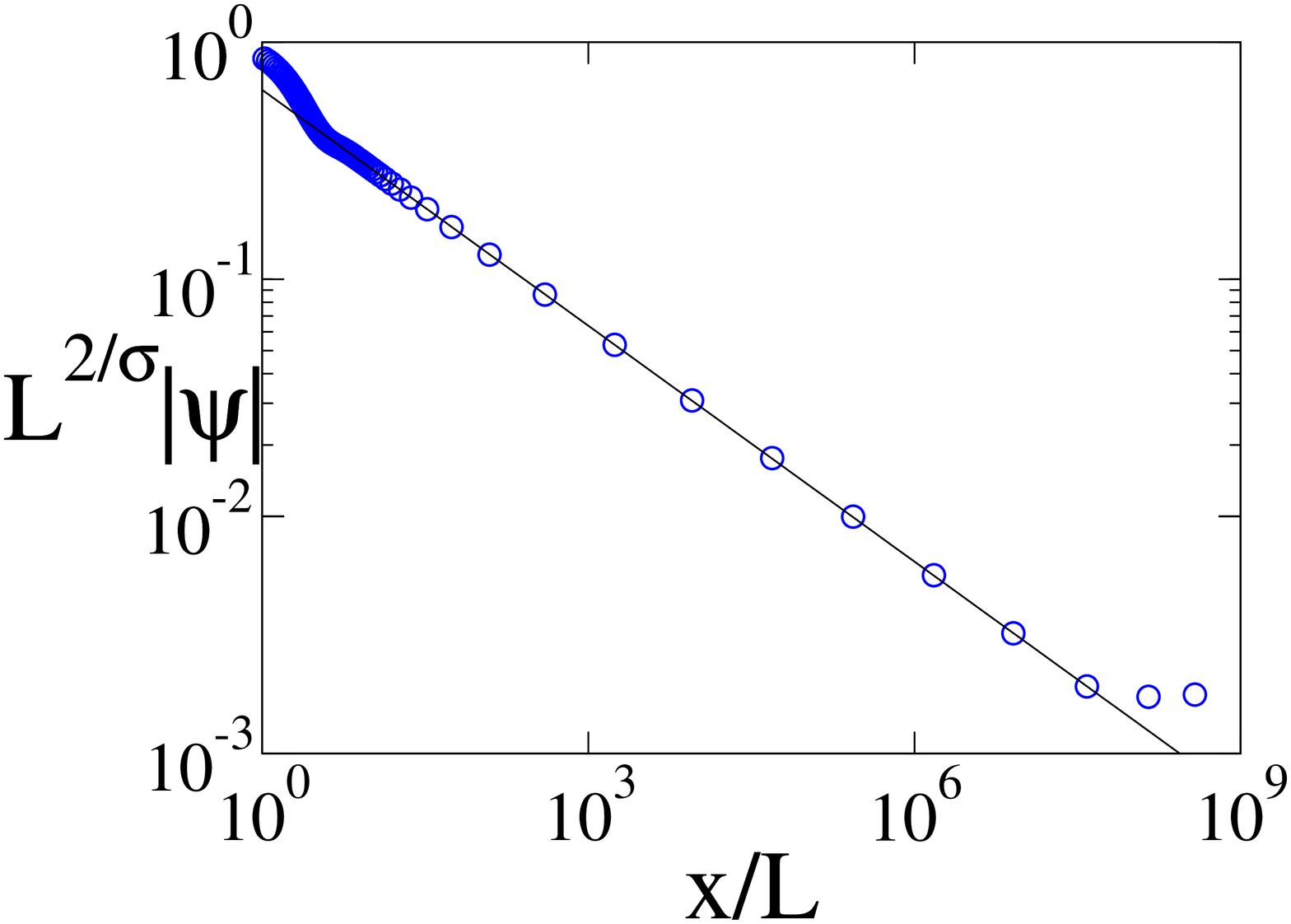}%
	}%
	\subfloat[$d=2,\sigma=3$]{\label{fig:supercrit_peak_rescaled_far_2D}%
		\includegraphics[clip,angle=-90,width=0.45\textwidth]%
			{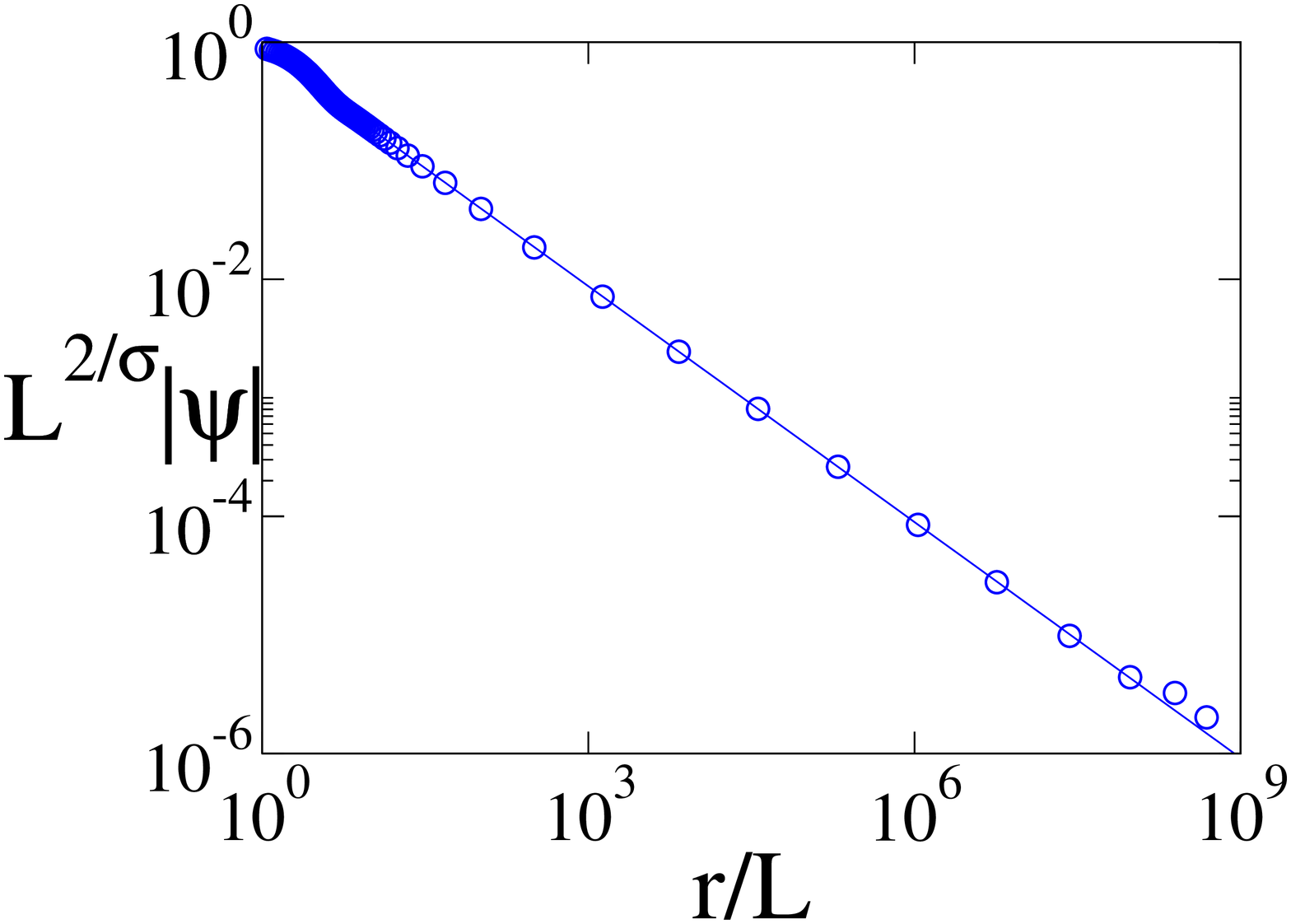}%
	}%

	\mycaption{\label{fig:supercrit_peak_rescaled_far}
		The solutions of Figure~\ref{fig:supercrit_peak_amp_Ls}, rescaled
		according to~\eqref{eq:psi_rescaled_peak}, at focusing the level~$1/L=10^8$
		(circles).
		Solid lines are the fitted curves~$y=0.63\cdot( r/L )^{-0.33}$ (left)
		and~$y=0.85\cdot( r/L )^{-0.66}$ (right).
	}
\end{figure}

\begin{figure}
\centering
	\subfloat[$d=1,\sigma=6$]{%
		\includegraphics[angle=-90,clip,width=0.5\textwidth]%
			{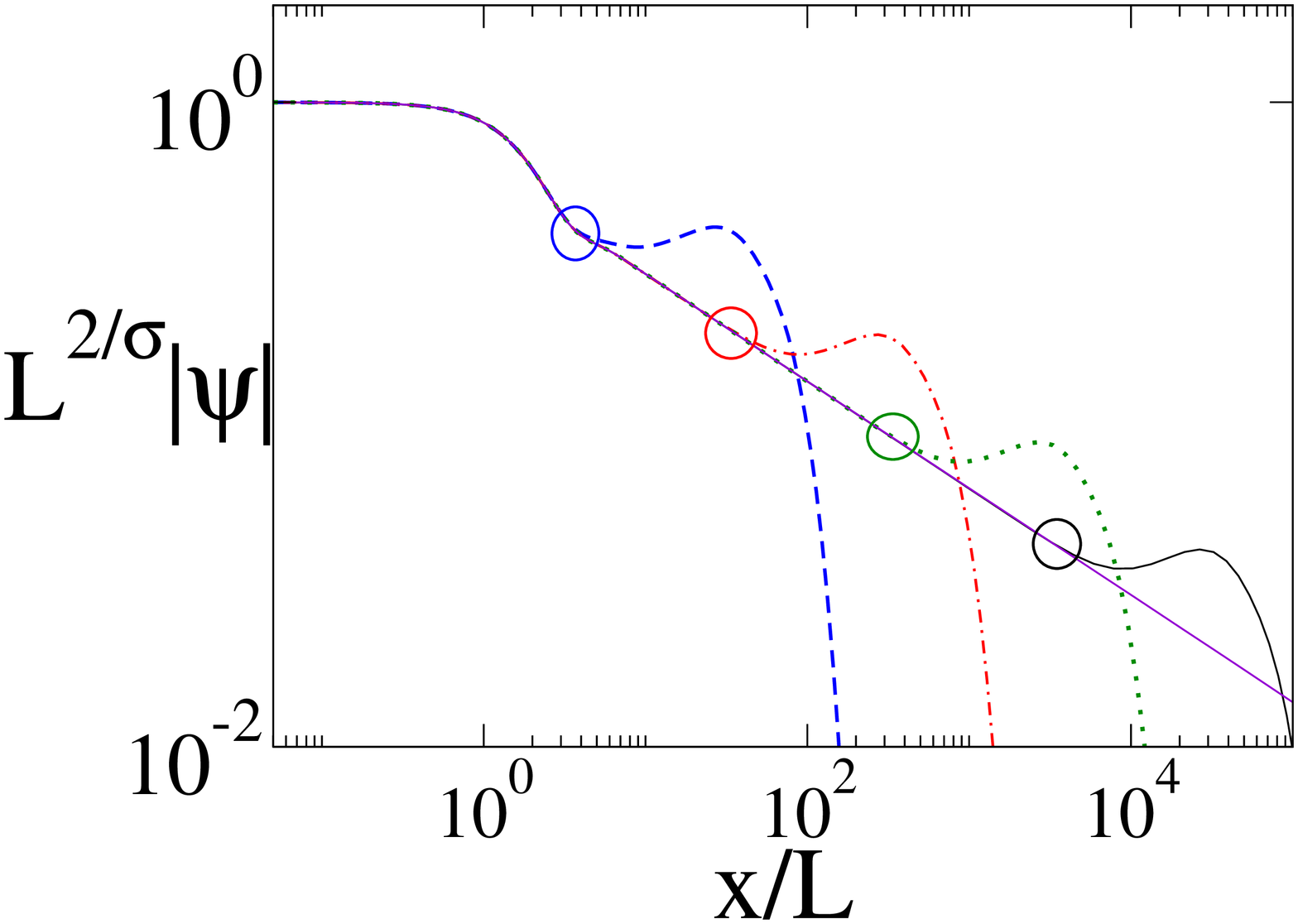}%
	}
	\subfloat[$d=2,\sigma=3$]{%
		\includegraphics[angle=-90,clip,width=0.5\textwidth]%
			{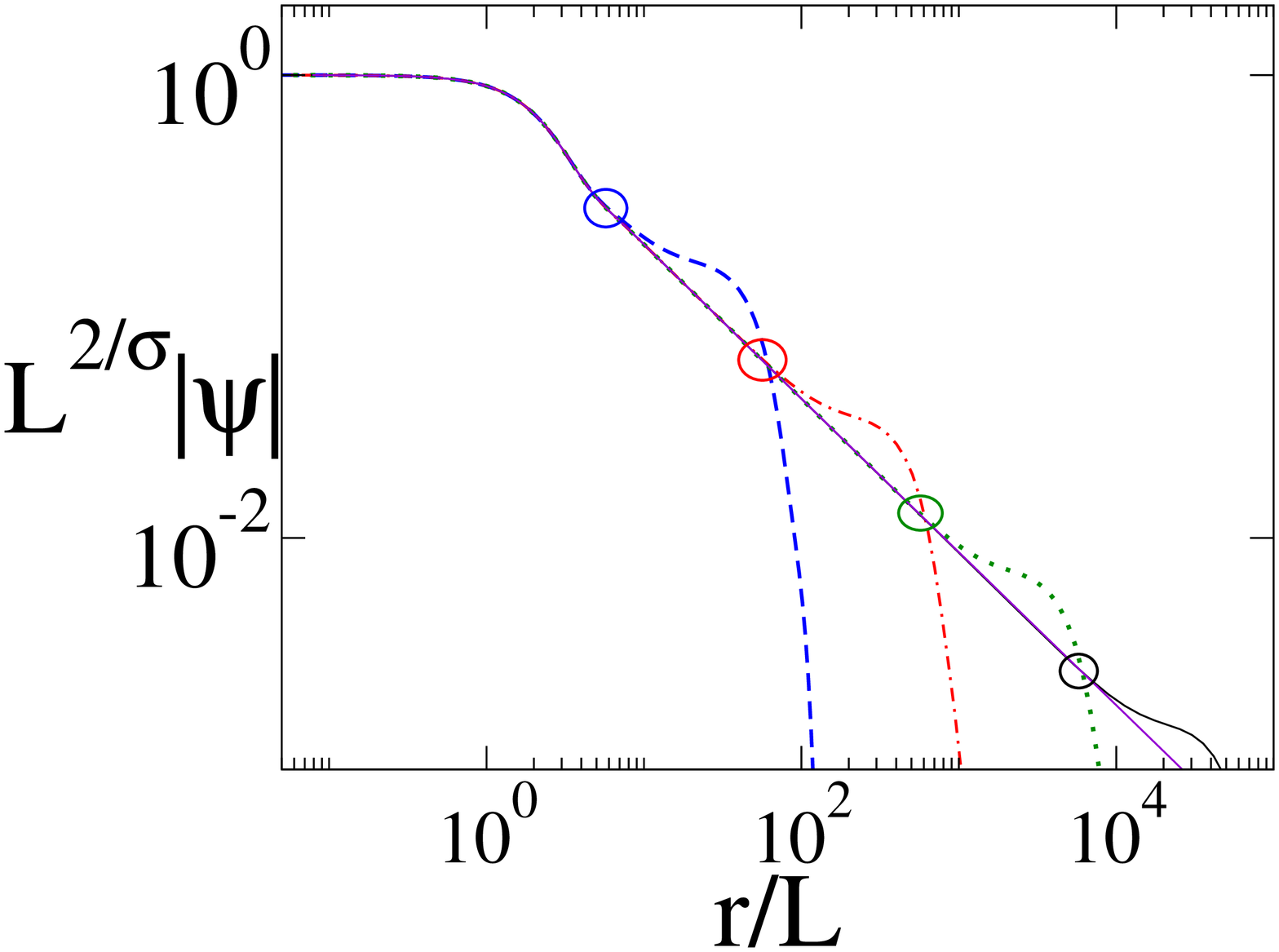}%
	}

	\mycaption{\label{fig:supercrit_peak_convergence}
		Convergence to a self-similar profile. 
		The solutions of Figure~\ref{fig:supercrit_peak_amp_Ls}, rescaled
		according to~\eqref{eq:psi_rescaled_peak}, as a function
		of~$\log(r/L)$, at the focusing levels  
			$L=10^{-1}$ (dashed blue line),
			$L=10^{-2}$ (dash-doted red line),
			$L=10^{-3}$ (dotted green line),
			$L=10^{-4}$ (solid black line)
			and 
			$L=10^{-8}$ (solid magenta line).
		The circles mark the approximate position where each curve bifurcates
		from the limiting profile, see also
		Table~\ref{tab:supercrit_peak_convergence}.
	}
\end{figure}
\begin{table}
	\centering
	\begin{tabular}{|c||c|c|c|c||c|} 
		\hline
		$1/L$ & $10$ & $100$ & $1000$ & $10000$ & $r_c$ \\
		\hline
		\hline
		$x/L~(d=1)$ &
			$3.6$ & $36$ & $360$ & $3600$ & $0.36$ \\
		\hline
		$r/L~(d=2)$ &
			$6$ & $60$ & $600$ & $6000$ & $0.6$\\
		\hline
	\end{tabular}

	\mycaption{\label{tab:supercrit_peak_convergence}%
		Position of circles in Figure~\ref{fig:supercrit_peak_rescaled_far}.
	}
\end{table}

\begin{figure}
\centering
	\subfloat[$d=1,\sigma=6$]{%
		\includegraphics[clip,width=0.35\textwidth,angle=-90]%
			{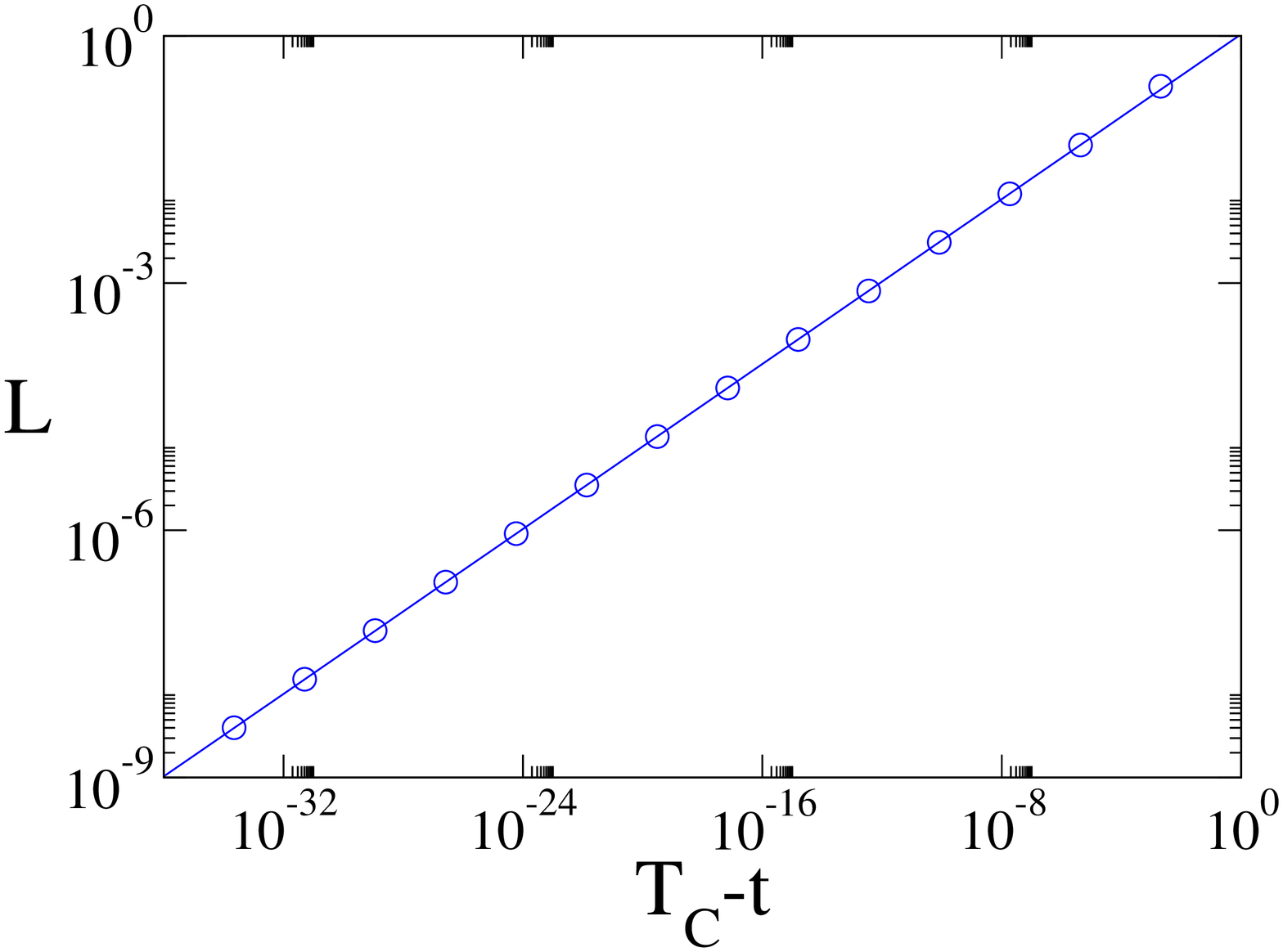}
	}
	\subfloat[$d=2,\sigma=3$]{%
		\includegraphics[clip,width=0.35\textwidth,angle=-90]%
			{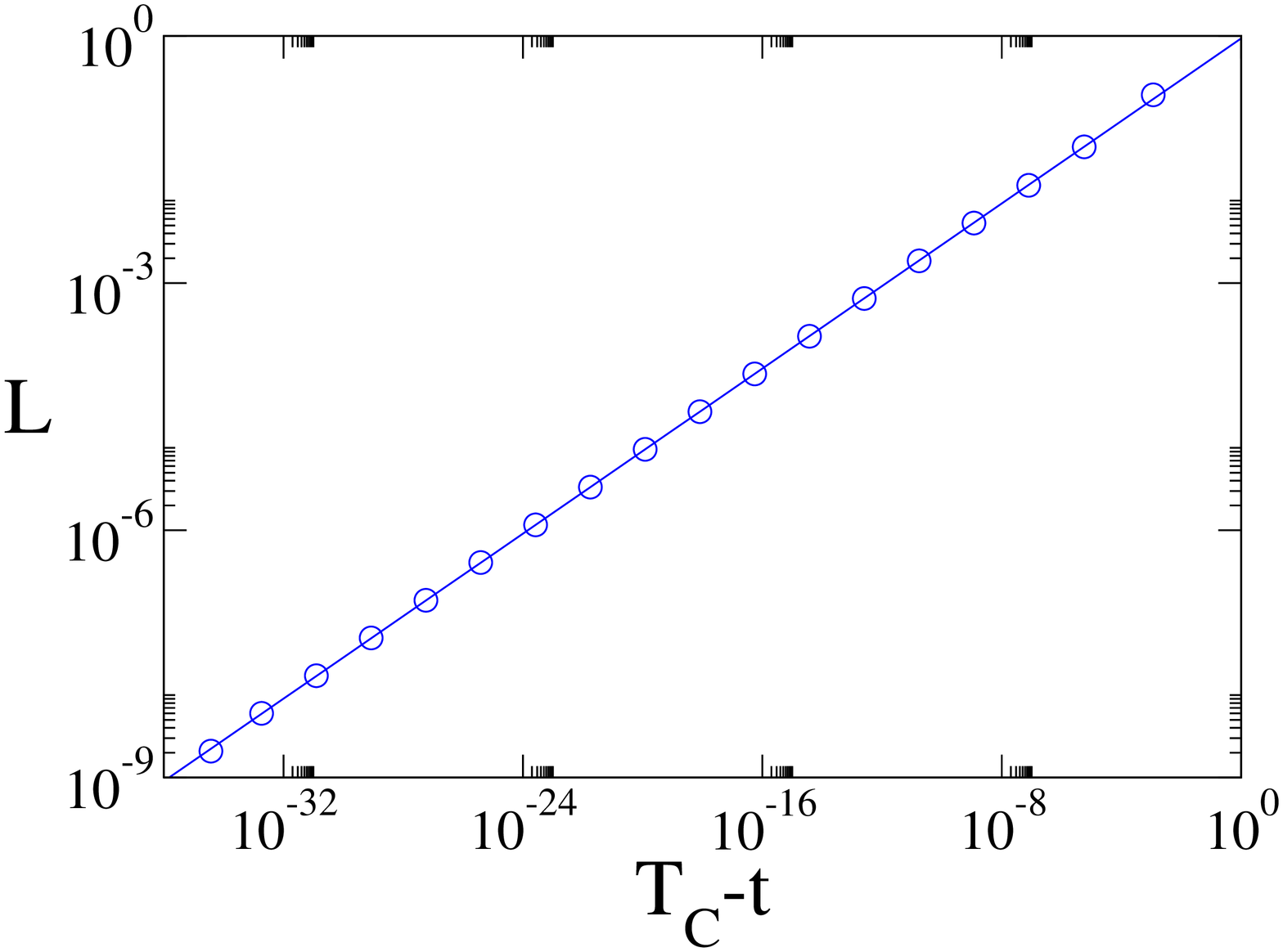}
	}

	\mycaption{\label{fig:supercrit_peak_powerlaw}
		$L(t)$ as a function of~$\left.(\TCrit-t)\right.$, on a logarithmic
		scale, for the solutions of Figure~\ref{fig:supercrit_peak_amp_Ls}
		(circles).
		Solid lines are the fitted
		curves~$\left.L=1.048\cdot(\TCrit-t)^{0.2502}\right.$ (A) 
		and~$\left.L=0.931\cdot(\TCrit-t)^{0.2504}\right.$ (B).
	}
\end{figure}

The radially-symmetric BNLS~\eqref{eq:radial_BNLS} was solved in the
supercritical case~$d=1,\sigma=6$ with the initial 
condition~$\psi_0(x)=1.6e^{-x^2}$, and in the supercritical 
case~$d=2,\sigma=3$ with the initial condition~$\psi_0(r)=3e^{-r^2}$.
In both cases, the solutions blowup at a finite time, see
Figure~\ref{fig:supercrit_peak_amp_Ls}.

To check whether the solutions collapse with the self-similar
profile~\eqref{eqs:supercrit_peak_QSS}, the solution was rescaled according
to~\eqref{eq:psi_rescaled_peak}.
The rescaled solutions at the focusing levels~$L=10^{-4}$ and~$L=10^{-8}$ are
indistinguishable in both the one-dimensional case
(Figure~\ref{fig:supercrit_peak_rescaled_1D}) and the two-dimensional case 
(Figure~\ref{fig:supercrit_peak_rescaled_2D}), providing numerical support that 
the solution collapses with the~$\psi_B$
profile~\eqref{eqs:supercrit_peak_QSS}.
As predicted, the self-similar profile is different than the ground-state~$R$.
Indeed, Figure~\ref{fig:supercrit_peak_rescaled_far} shows that as~$\rho\to\infty$,
the self-similar profile of~$\psi$ decays as~$\rho^{-2/\sigma}$, which is in agreement
with the decay rate of~$B_1(\rho)$.

We next verify that the solution converges to the asymptotic profile
for~$r\in[0,r_c]$, i.e., for~$\rho\in[0,r_c/L(t)]$.
To do this, we plot in Figure~\ref{fig:supercrit_peak_convergence} the rescaled 
solution at focusing levels of~$1/L=10,100,1000,10000$, as a function
of~$\log(r/L)$.
The curves are indistinguishable at~$r/L=\mathcal{O}(1)$, but bifurcate at
increasing values of~$r/L$.
These ``bifurcations positions'' are marked by circles in
Figure~\ref{fig:supercrit_peak_convergence}, and their $r/L$ values are listed in
Table~\ref{tab:supercrit_peak_convergence}.
The ``bifurcation positions'' are linear in~$1/L$, indicating that the region
where~$\psi\sim\psi_B$ is indeed~$\rho\in[0,r_c/L(t)]$, which corresponds
to~$r\in[0,r_c]$.

To compute the blowup rate~$p$,  we performed a least-squares fit of~$\log(L)$
with $\log(T_c-t)$, see Figure~\ref{fig:supercrit_peak_powerlaw}.
The resulting values are~$p\approx 0.2502$ in the~$d=1,\sigma=6$ case 
and~$p\approx 0.2504$ in the~$d=2,\sigma=3$ case.
%
Next, we provide two indications that the blowup rate is exactly~$1/4$, i.e.,
that
\[
	L(t)\sim \kappa\sqrt[4]{\TCrit-t},\qquad \kappa>0.
\]
First, if the blowup rate is exactly a quartic root, then~$
	L^3L_t\to -\frac{\kappa^4}{4}<0.
$
Indeed, Figure~\ref{fig:supercrit_peak_L3Lt} shows that 
in the case~$d=1$,~$\sigma=6$,~$L^3L_t\to -0.289$, implying that 
\begin{equation}	\label{eq:kappa_d1s6}
	\kappa(d=1,\sigma=6) \approx \sqrt[4]{4\cdot0.289} \approx 1.037\,.
\end{equation}
In the case~$d=2$,~$\sigma=3$,~$L^3L_t\to -0.171$, implying that 
\begin{equation}	\label{eq:kappa_d2s3}
	\kappa(d=2,\sigma=3) \approx \sqrt[4]{4\cdot0.171} \approx 0.909\,.
\end{equation}
Since~$L^3L_t$ converges to a finite, negative constant, this shows that the
blowup rate is exactly~$1/4$.

Second, according to Lemma~\ref{lem:supercrit_peak_rate_profile},
if~$L^3L_t\to 0$ the self-similar profile~$B(\rho)$ should not satisfy the
standing-wave equation~\eqref{eq:stationary_state}.
In Figure~\ref{fig:supercrit_peak_rescaled} we saw that the rescaled self-similar
BNLS solutions and the corresponding ground-states differ considerably
(compare with Figure~\ref{fig:crit_peak_rescaled}), indicating that the blowup
rate is exactly~$1/4$.
Therefore, the numerical results again support
Conjecture~\ref{conj:supercrit_peak_rate_profile}.

Finally, we verified that the value of~$\kappa$ in the blowup rate~\eqref{eq:rate_14}
is universal.
We solve the BNLS in the case $d=1,\sigma=6$ with the initial
condition~$\psi_0(x)=2e^{-x^4}$.
In this case, 
the calculated value of~$\kappa(d=1,\sigma=6)$ is~$
	\kappa=\displaystyle\lim_{t\to\TCrit}\sqrt[4]{-4L_tL^3}\approx1.037
$, which is equal, to first~$3$ significant digits, to the previously obtained
value, see~\eqref{eq:kappa_d1s6}, for the initial condition~$\psi_0(x)=1.6e^{-x^2}$.
Similarly, in the case $d=2,\sigma=3$, we solve the equation with the initial
condition~$\psi_0(x)=3e^{-x^4}$. 
The calculated value of~$\kappa(d=2,\sigma=3)$ is~$
	\kappa=\displaystyle\lim_{t\to\TCrit}\sqrt[4]{-4L_tL^3}\approx0.913
$, which is equal, to first~$2$ significant digits, to the previously obtained
value, see~\eqref{eq:kappa_d2s3}, for the initial condition~$\psi_0(x)=3e^{-x^2}$.

\section{\label{sec:num_meth}Numerical methods}

\subsection{\label{ssec:SGR}Adaptive mesh construction using the SGR method}

In this study, we computed singular solutions of the BNLS 
equation~\eqref{eq:radial_BNLS}.
These solutions become highly-localized, so that the spatial
scale-difference between the singular region~$r-\rmax = {\cal O}(L)$
and the exterior regions can be as large as~${\cal O}(1/L)\sim
10^{10}$. In order to resolve the solution at both the singular and
non-singular regions, we use an adaptive grid.

We generate the adaptive grids using the {\em Static Grid Redistribution}~(SGR)
method, which was first introduced by Ren and Wang~\cite{Ren-00}, and later
simplified and improved by Gavish and Ditkowsky~\cite{SGR-08}.
Using this approach, the solution is allowed to propagate (self-focus) until it
becomes under-resolved.  
At this stage, a new grid, with the same number of grid-points, is generated
using De'Boors `equidistribution principle', wherein the grid points 
$\{r_m\}$ are spaced such that a certain weight function~$w_1[\psi]$ is
equidistributed, i.e., that \[
	\int_{r=r_m}^{r_{m+1}} w_1\left[ \psi(r) \right]dr
	= \text{const},
\]
see~\cite{Ren-00,SGR-08} for details.

\begin{figure}
	\centering
	\subfloat[old method,~$L=10^{-6}$]{\label{fig:w3_effect_old}%
		\includegraphics[clip,width=0.45\textwidth]%
			{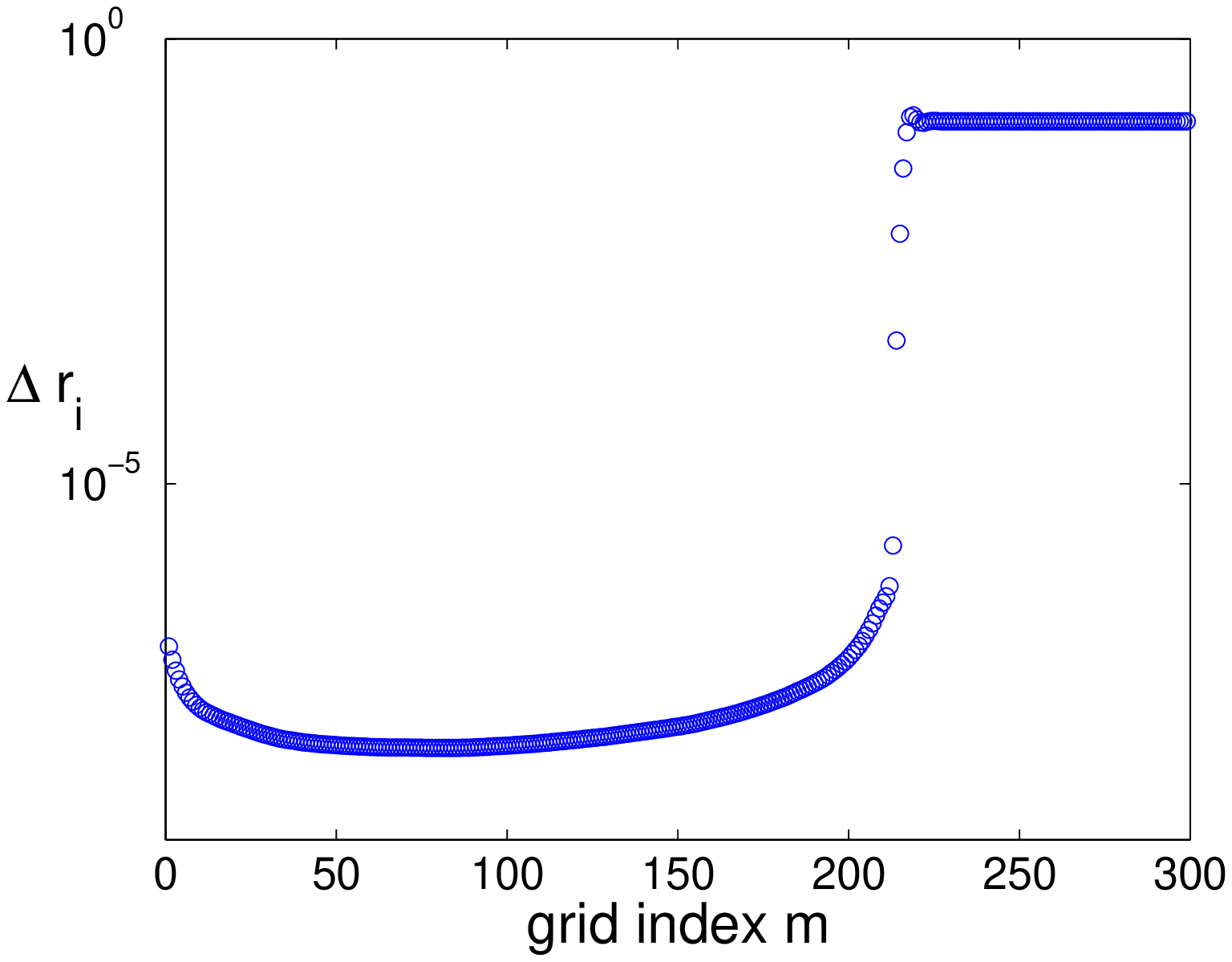}%
	}
	\subfloat[new method,~$L=10^{-12}$]{\label{fig:w3_effect_new}%
		\includegraphics[clip,width=0.45\textwidth]%
			{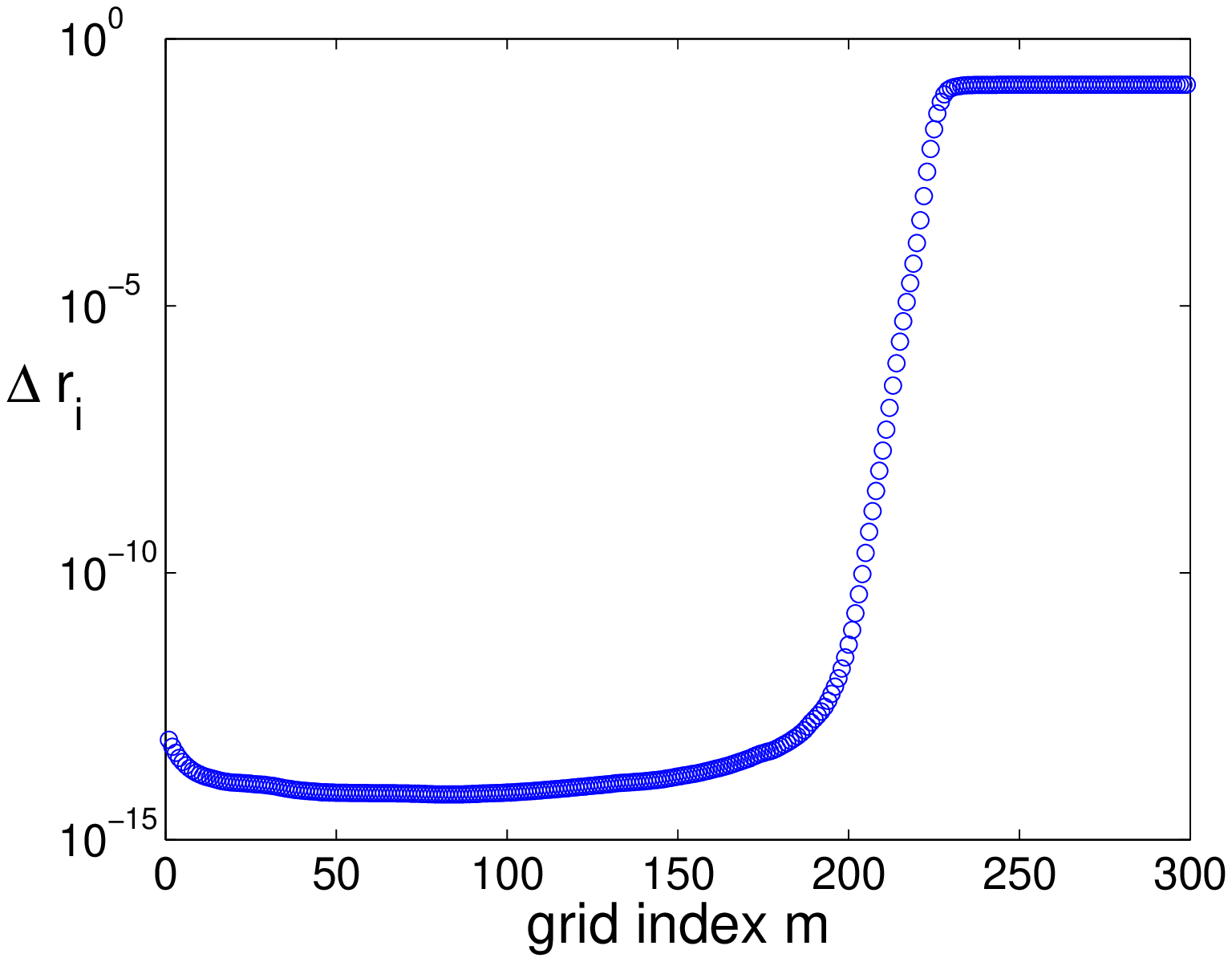}%
	}

	\mycaption{\label{fig:w3_effect}%
		The grid-spacing~$\Delta r_m$ obtained using the SGR method
		of~\cite{SGR-08} for a peak-type singular solution of the BNLS.
		A) The grid generated the original method of~\cite{SGR-08} at focusing
		level of~$L=10^{-6}$. 
		The Singular and non-singular regions are well-resolved, but the
		transition region~$\Delta r_m$ displays a discontinuity.
		At this point, the finite difference operator becomes ill-conditioned.
		B) same as (A), after adding the new penalty function~$w_3$, at focusing
		level~$L=10^{-12}$.
		Even at this much larger focusing level, the transition region is now
		well resolved.
	}
\end{figure}

The algorithms of~\cite{Ren-00,SGR-08} keeps a recursive set of grids, the
coordinates of each is given in the reference frame of the previous grid.
This enables the simulation to reach high focusing levels, where the grid-points
position cannot be stored at the physical reference frame due to loss of
significant digits.
In this study, however, we implement a simplified version of the method
of~\cite{SGR-08}, in which we dispense with the above hierarchy of grids and
store the grid-points position in terms of the original reference frame.
On this highly non-uniform grid, we use a standard non-uniform finite-difference
approximation of the radial biharmonic operator~\eqref{eq:radial_bi_Laplacian}.
The approximation is third-order accurate, with a seven-point stencil.
Thus, we are limited by the standard machine accuracy\footnote{
	At the time of writing, the standard machine accuracy is IEEE~$64$-bit
	floating points, with relative machine error of~$10^{-16}$.
	Using~$128$-bit floating points will, for all practical purposes, eliminate the
	above restriction.
} to focusing levels of no more than~$L\approx 10^{-12}$.
However, the resulting gain in software (and numerical) simplicity justifies
this limitation.

The method in~\cite{SGR-08} allows control of the fraction of grid points
that migrate into the singular region, preventing under-resolution at the
exterior regions.
This is done by using a weight-function~$w_2$, which penalizes large inter-grid
distances.
However, we found that this numerical mechanism, while necessary, is
insufficient for our purposes.
In order to understand the reason, let us consider the grid-point spacings~$
	\Delta r_m = r_{m+1}-r_m
$.
Using the method of~\cite{SGR-08} with both~$w_1$ and~$w_2$ causes a very sharp 
bi-partition of the grid points -- to those inside the singular region, whose
spacing is determined by~$w_1$ and is~$\Delta r_m=\mathcal{O}(L)$,
and to those outside the singular region, whose spacing is determined by~$w_2$ 
and is~$\Delta r_m=\mathcal{O}(1)$, see Figure~\ref{fig:w3_effect_old}.
Inside each of these regions, the finite difference approximation we use is well
conditioned.
However, at the transition between these two regions, the finite-difference stencil,
seven-points in width, spans grid-spacings with~$\mathcal{O}(1/L)$
scale-difference --- leading to under-resolution which completely violates the
validity of the finite-difference approximation.

In order to overcome this limitation, we improve the algorithm of~\cite{SGR-08}
by adding a third weight function \[
w_3 (r_m) = \sqrt{	1+ \frac{\abs{\Delta^2r_m}}{\Delta r_m} },
\]
which penalizes the second-difference~$\Delta^2 r_m = \Delta r_{m+1}-\Delta r_m$
operator of the grid locations, allowing for a smooth transition between the
singular region and the non-singular region, see Fig~\ref{fig:w3_effect_new}.

On the sequence of grids, the equations are solved using a
Predictor-Corrector Crank-Nicholson scheme, which is second-order in time.

\subsection{\label{ssec:SRM}The Spectral Renormalization Method}

Here, we describe the adaptation of the Spectral Renormalization method (SRM)
to the standing-wave BNLS equation~\eqref{eq:StandBNLS}.
Denoting the Fourier transform of~$R(\bvec{x})$ by~$\mathcal F[R](\bvec{k})$,
equation~\eqref{eq:StandBNLS} transforms to 
\begin{equation}	\label{eq:SRM_FT}
	\mathcal F[R](\bvec{k})=\frac{1}{k^4+1}
		\mathcal F\left[|R|^{2\sigma}R\right],
\end{equation}
where~$k^4=|\bvec{k}|^4$, leading to the fixed-point iterative scheme \[
	\mathcal F[R_{m+1}]=
		\frac{1}{k^4+1} \mathcal F\left[|R_m|^{2\sigma}R_m\right],
		\qquad m=0,1,\dots~.
\]
Typically this iterative scheme diverges either to~$\infty$ or to~$0$.
In order to avoid this problem, we renormalize the solution as follows.
Multiplying equation~\eqref{eq:SRM_FT} by~$\mathcal F[R]^*$ and integrating
over~$\bvec{k}$ gives the integral relation:
\begin{subequations}	\label{eqs:SRM}
	\begin{equation} 	\label{eq:SRM_integral_relation}
		SL=SR,\quad\text{where}\quad
		SL[R] \equiv \int\abs{\mathcal F[R]}^2d\bvec{k},	\quad 
		SR[R] \equiv \int \frac{1}{k^4+1} 
				\mathcal F\left[|R|^{2\sigma}R\right]\mathcal F[R]^*d\bvec{k}. 
	\end{equation}
	We now define~$R_{m+\frac{1}{2}}=C_m R_m$ such that the integral 
	relation~\eqref{eq:SRM_integral_relation} 	is satisfied by 
	$R_{m+\frac{1}{2}}$, i.e, that \[
		SL\left[R_{m+1/2}\right] = C_m^2SL[R_m]
		= C_m^{2\sigma+2}SR[R_m]=SR\left[R_{m+1/2}\right],
	\]
	leading to~$
		C_m = \left(\frac{SL[R_m]}{SR[R_m]}\right)^{\frac{1}{2\sigma}},
	$ and hence to~$$
		\abs{R_{m+1/2}}^{2\sigma}R_{m+1/2} =
		\left(\frac{SL[R_m]}{SR[R_m]}\right)^{1+\frac{1}{2\sigma}}
		\abs{R_m}^{2\sigma}R_m.
	$$
	The Spectral Renormalization method is therefore given by the iterations
	\begin{equation}
		\mathcal F(R_{m+1})=
			\left(\frac{SL[R_m]}{SR[R_m]}\right)^{1+\frac{1}{2\sigma}}
				\frac{1}{k^4+1} \mathcal F\left(|R_m|^{2\sigma}R_m\right),
		\qquad m=1,2\dots~.
	\end{equation}
\end{subequations}

In this work, we use the SRM to solve~\eqref{eq:StandBNLS} for the
cases~$d=1,2,3$ without imposing radial symmetry.
Alternatively, one might have solved the radial
equation~\eqref{eqs:radial_stationary_state} using a modified Hankel-like
transform instead of the Fourier Transform.
Our main reason for not doing so is the convenience and cost-effectiveness of
using the Fast Fourier Transform.
We also note that our non-radially-symmetric method produced a
radially-symmetric solution, which suggests that the ground state is radially
symmetric.


\subsection*{Acknowledgments} 
We thank Nir Gavish for useful discussions.
This research was partially supported by grant \#123/2008 from the Israel
Science Foundation (ISF).

\appendix

\section{\label{app:compactness}Compactness Lemma}
Here we provide an extension of the Compactness Lemma for~$H_{\rm radial}^{1}$
functions~\cite{Strauss-77}, to the case of~$H^2_{\rm radial}$:

\begin{lem}{(Compactness Lemma)}\label{lem:compactness}
Let~$d\geq 2$ and let~$\sigma>0$ be in the $H^2$-subcritical
regime~\eqref{eq:admissible-range}.
Then, the embedding~$
	\left.
		H_{\rm radial}^2(\mathbb{R}^d)
			\to L^{2(\sigma+1)}(\mathbb{R}^d)
	\right.
$ is compact, i.e., every bounded sequence~$
	u_{n'}\in H_{\rm radial}^2(\mathbb{R}^d)
$ has a subsequence~$u_n$ which converges strongly in~$
	L^{2(\sigma+1)}(\mathbb{R}^d)
$.

\begin{proof}
If~$\norm{u_{n'}}_{H^2}\leq M$, then the sequence~$u_{n'}$ has a subsequence
$u_n$ which converges weakly to~$u$ in~$H^2$. 
Since the limit of radial functions is a radial function,~$u\in H_{\rm radial}^2$. 
In addition, since for any bounded domain~$\Omega$, the embedding 
$H^2(\Omega)\to L^2(\Omega)$ is compact, there is a subsequence which 
converges strongly to~$u$ in~$L^2(\Omega)$, i.e.,~$
	{ \lim_{n\to\infty}} \int_{\Omega}
		\left|u_n-u\right|^2d\bvec{x}=0
$.
From the Gagliardo-Nirenberg inequality on the bounded domain~$\Omega$, 
see~\cite{Gagliardo-58,Gagliardo-59,Nirenberg-59}, \[
	\norm{f}_{L^{2(\sigma+1)}(\Omega)}^{2(\sigma+1)}  
	\leq B_{\sigma,d,\Omega} 
		\norm{\Delta f}_{L^2(\Omega)}^{\sigma d/2} \cdot 
		\norm{f}_{L^2(\Omega)}^{2(\sigma+1)-\sigma d/2}
	= B_{\sigma,d,\Omega}
		\norm{\Delta f}_{L^2(\Omega)}^{\sigma d/2} \cdot 
		\norm{f}_{L^2(\Omega)}^{2(1-\sigma\frac{d-4}4)}
\]
and since~$1>\sigma\frac{d-4}4$ in the~$H^2$-subcritical case, it follows that
$u_n\to u$ strongly in~$L^{2(\sigma+1)}(\Omega)$, so that \[
	\lim_{n\to\infty} \int_{\Omega}
		\left|u_n-u\right|^{2(\sigma+1)} d\bvec{x} = 0.
\]

Next, Strauss radial Lemma~\cite{Strauss-77} for~$H^{1}$ functions gives that
$\forall\rho_{\epsilon}>1$ and~$n$, \[
	\int_{\left|x\right|>\rho_{\epsilon}}
		\left|u_n\right|^{2(\sigma+1)}d\bvec{x}
	\leq  
		\frac{C}{\rho_{\epsilon}^{(d-1)\sigma}},
\] so that~$\forall\epsilon\exists\rho_{\epsilon}$ s.t.~$\forall n$ \[
	\int_{|x|>\rho_{\epsilon}} 
		\left|u_n\right|^{2(\sigma+1)}d\bvec{x} \leq \epsilon.
\]

Finally, since \[
	\norm{u_n-u}_{L^{2(\sigma+1)}(\mathbb{R}^d)}  
	\leq  \norm{u_n-u}_{
		L^{2(\sigma+1)}(|\bvec{x}|<\rho_{\epsilon})
	} 
	+ \norm{u_n}_{
		L^{2(\sigma+1)}(|\bvec{x}|>\rho_{\epsilon})
	} 
	+ \norm{u}_{
		L^{2(\sigma+1)}(|\bvec{x}|>\rho_{\epsilon})
	}
\]
the convergence in~$\mathbb{R}^d$ is obtained. 

\end{proof}\end{lem} 

\section{\label{app:proof-power-conc-lower-bound}%
	Proof of Corollary~\ref{cor:power-conc-lower-bound} }

\begin{proof}
	From the proof of Theorem~\ref{thrm:self-similarity}, \[
		\norm{\psi(t_k,r)}_{L^2(r<a(t_k))}^2  
		= \norm{\phi_k(r)}_{L^2(r<a(t_k)/\LN(t_k))}^2,
	\]
	and since~$\lim_{k\to\infty}a(t_k)/\LN(t_k)=\infty$, we have that
	$\forall M>0$
	\[
		\liminf_{k\to\infty}\norm{\phi_k(r)}_{L^2(r<M)}^2
		\leq 
		\liminf_{k\to\infty}\norm{\phi_k(r)}_{L^2(r<a(t_k)/\LN(t_k))}^2.
	\]
	Since~$\phi_k(r)\underset{L^2}{\weakto}\Psi$, it follows that
	$\phi_k(r)\underset{L^2(M)}{\weakto}\Psi$, and so \[
		\norm{\Psi}_{L^2(r<M)}^2  
		\leq	\liminf_{k\to\infty}\norm{\phi_k(r)}_{L^2(r<M)}^2
	\]
	This is true~$\forall M$, and so \[
		\BPCrit \leq\norm{\Psi}_{L^2}^2  \leq  
		\liminf_{k\to\infty}\norm{\psi(t_k,r)}_{L^2(r<a(t_k))}^2.
	\]

	For the second result, since~$\norm{\Psi}_{L^2}\geq\BPCrit$, it follows that
	for all~$\epsilon>0$ there exist~$K>0$ such that \[
		\norm{\Psi}_{L^2(r>K)}\geq(1-\epsilon)\BPCrit.
	\]
	Therefore, since \[
		\norm{\psi(t_k,r)}_{L^2(r<K\cdot \LN(t_k))}^2  
		= \norm{\phi_k(r)}_{L^2(r<K)}^2,
	\]
	a similar argument as in the previous section gives \[
		(1-\epsilon)\BPCrit \leq\norm{\Psi}_{L^2(r<K)}^2  \leq  
		\liminf_{k\to\infty}\norm{\psi(t_k,r)}_{L^2(r<K\cdot \LN(t_k))}^2.
	\]
\end{proof}

\section{\label{app:gagli-niren}Gagliardo Nirenberg inequality for~$H^2$ functions}

In $L^2$-critical case~$\sigma d=4$, which is always in the~$H^2$-subcritical
regime~\eqref{eq:admissible-range}, the appropriate Gagliardo-Nirenberg inequality
in~$H^2$ is~\cite{Gagliardo-58,Gagliardo-59,Nirenberg-59}:

\begin{lem}{(Gagliardo-Nirenberg inequality)}
Let~$\sigma d=4$, and let~$f\in H^2(\Real^d)$, then 
\begin{equation}	\label{eq:Critical_Gagliardo}
	\norm{f}_{2(\sigma+1)}^{2(\sigma+1)}  
	\leq  B_{\sigma,d}
		\norm{\Delta f}_2^2\norm{f}_2^{2\sigma}.
\end{equation}
\end{lem}

We note that the ground-state~$R$ of equation~\eqref{eq:stationary_state} is the 
minimizer of the Gagliardo-Nirenberg
inequality~\cite{Fibich_Ilan_George_BNLS:2002}, and that its $L^2$ norm, the
critical power, satisfies \[
	\BPCrit = \norm{R}_2^2 = \left( \frac{\sigma+1}{B_{\sigma,d}}
	\right)^{1/\sigma}.
\]
Hence, the Gagliardo-Nirenberg inequality implies the following Corollary:
\begin{cor}\label{cor:P2H_bound} Let~$f\in H^2$ and~$\sigma d=4$, then \[
		H[f]  \geq  \left[
		1-\left( \frac{\norm{f}_2^2}{\BPCrit} \right)^{\sigma}
		\right] 
		\cdot \norm{\Delta f}_2^2,
	\]
	 so that 
	\begin{equation}	\label{eq:P2H_bound}
		\norm{f}_2^2\leq\BPCrit  \implies  H[f]\geq0.
	\end{equation}
\end{cor}

\section{\label{app:WKB}WKB analysis of eq.~\eqref{eq:supercrit_peak_ODE}.}

As $\rho\to\infty$, the nonlinear term in~\eqref{eq:supercrit_peak_ODE} becomes
negligible, and~\eqref{eq:supercrit_peak_ODE} reduces to
\begin{equation}	\label{eq:SC_WKB_profile}
	-B(\rho) 
	-\Delta_\rho B 
	+ i b^3 \left(
		\frac{2}{\sigma}B + \rho B_\rho 
	\right)
	 = 0,
	 \qquad b^3 = \frac{\kappa^4}{4},
\end{equation} where
\[
	\Delta_{\rho}^2 = 
	-\frac{(d-1)(d-3)}{\rho^3}\partial_{\rho}
		+\frac{(d-1)(d-3)}{\rho^2}\partial_{\rho}^2
		+\frac{2(d-1)}{\rho}\partial_{\rho}^3
		+\partial_{\rho}^4 ~ .
\]

In order to apply the WKB method, we substitute~$B(\rho) = \exp(w(\rho))$, and expand
\[
	w(\rho) \sim w_0(\rho) + w_1(\rho) + \dots\; .
\]
Substituting~$w_0(\rho) = \alpha \rho^p$ and balancing terms shows that~$p=4/3$,
and that the equation for the leading-order, 
the~$\mathcal{O}\left( \rho^{4/3} \right)$ terms, is 
\[
	\left( w_0^\prime \right)^3 
		= \left( \frac{4\alpha}{3} \right)^3 \rho
		= ib^3 \rho.
\] 
Therefore,
\[
	\alpha = \frac{3}{4}b e^{i\frac \pi 6 + i\pi \frac{2k}{3}}
	=
	\frac 3 4 b \cdot
	\left\{
		-i ,
		\frac{\sqrt{3}+i}2 ,
		\frac{-\sqrt{3}+i}2 
	\right\} .
\]
The equation for the next order, the~$\mathcal{O}\left( 1 \right)$ terms, is \[
	1+2dib^3 = ib^3 \left( \frac 2\sigma + 3\rho w_1^\prime \right), 
\] implying that \[
w_1 = \frac 13 \left( \frac 2\sigma(1-\sigma d)
	- \frac{1}{ib^3} \right) \log \rho .
\]
The next-order terms are~$\mathcal{O}\left( \rho^{-4/3} \right)=o(1)$ and can be
neglected.
\begin{subequations} \label{eqs:WKB_solutions}
We therefore obtain the three solutions 
	\begin{eqnarray} 
		B_2(\rho) &\sim&
			\frac{1}{\rho^{\frac{2}{3\sigma}(\sigma d-1)}}
			\exp\left( 
				-i\frac{3}{4}b\rho^{4/3}
				-i\frac{1}{3b^3} \log(\rho)
			\right), \\
		B_3(\rho) &\sim&
			\frac{
				\exp\left( 
				\frac{3\sqrt{3}}{8}b \rho^{4/3}
				\right)
			}{\rho^{\frac{2}{3\sigma}(\sigma d-1)}}
			\exp\left( 
				+i\frac{3}{8}b\rho^{4/3}
				-i\frac{1}{3b^3} \log(\rho)
			\right), \\
		B_4(\rho) &\sim&
			\frac{
				\exp\left( - ~
					\frac{3\sqrt{3}}{8}b \rho^{4/3}
				\right)
			}{\rho^{\frac{2}{3\sigma}(\sigma d-1)}}
			\exp\left( 
				+i\frac{3}{8}b\rho^{4/3}
				-i\frac{1}{3b^3} \log(\rho)
			\right) .
	\end{eqnarray}
	Since~\eqref{eq:SC_WKB_profile} is a fourth order ODE, another solution is
	required.
	To obtain the fourth solution, we substitute~$w_0\sim \beta\log(\rho)$  
	in~\eqref{eq:SC_WKB_profile} and obtain that the equation for the
	leading-order, the~$\mathcal{O}(1)$ terms, is \[
		-1+ib^3 \left( \frac 2\sigma+\beta \right) = 0 ,
	\]
	and that the next order is~$o(1)$ and can be neglected.
	The fourth solution is therefore 
	\begin{equation} 
		B_1(\rho) \sim
			\rho^{-\frac 2\sigma -i\frac{1}{b^3}}.
	\end{equation}
\end{subequations}

\section{\label{app:Noether}Application of Noether Theorem for the BNLS}

The Lagrangian density of the BNLS is  \[
	\mathcal{L}\left(\psi, \psi^*, \psi_t, \psi_t^*,
		\Delta\psi, \Delta\psi^* \right)
		= \frac{i}{2} \left(
			\psi_t\psi^* - \psi_t^*\psi
		\right)
		- \left|\Delta \psi\right|^2
		+ \frac{1}{1+\sigma} \left|\psi\right|^{2(\sigma+1)}.
\]

We cite here Noether's Theorem, as given in \cite{Sulem-99}:
\begin{thrm}[Noether Theorem] \label{thrm:Noether}
	If the action integral $
	\iint \mathcal{L}\,
		d\bvec{x}dt
	$ is invariant under the infinitesimal transformation
	\begin{eqnarray*}
		t &\mapsto& \tilde{t} = t+\delta t(\bvec{x},t,\psi), \\
		\bvec{x} &\mapsto& \tilde{\bvec{x}} = \bvec{x}
			+\delta \bvec{x}(\bvec{x},t,\psi), \\
		\psi &\mapsto& \tilde \psi = \psi+\delta \psi(\bvec{x},t,\psi),
	\end{eqnarray*}
	then 
	\begin{equation}	\label{eq:Noether}
		\int \left[ 
			\frac{\partial \mathcal{L}}{\partial \psi_t} 
				\left( 
					\psi_t \delta t
					+ \nabla \psi \cdot \delta \bvec{x}
					-\delta \psi
				\right)
				+
			\frac{\partial \mathcal{L}}{\partial \psi_t^*} 
				\left( 
					\psi_t^* \delta t
					+ \nabla \psi^* \cdot \delta \bvec{x}
					-\delta \psi^*
				\right)
			-\mathcal{L}\delta t
		\right]
		d\bvec{x}
	\end{equation}
	is a conserved quantity.
\end{thrm}

For example, the BNLS action integral is invariant under the phase-multiplication~$
	\psi(t,\bvec{x}) \mapsto
	e^{i\varepsilon} \psi(t,\bvec{x}).
$
In this case,~$
	\delta t = 0, \delta \bvec{x} = 0,
	\delta \psi = i \psi,
$
and so Theorem~\ref{thrm:Noether} implies that the integral
\begin{align*}
		& \int \left[ 
			\frac{\partial \mathcal{L}}{\partial \psi_t} 
				\left( 
					\psi_t \delta t
					+ \nabla \psi \cdot \delta \bvec{x}
					-\delta \psi
				\right)
				+
			\frac{\partial \mathcal{L}}{\partial \psi_t^*} 
				\left( 
					\psi_t^* \delta t
					+ \nabla \psi^* \cdot \delta \bvec{x}
					-\delta \psi^*
				\right)
			-\mathcal{L}\delta t
		\right]
		d\bvec{x} \\
		& =
		\int \left[ 
			\frac i2\psi^* 
				\left( 
					\psi_t \cdot 0
					+ \nabla \psi \cdot 0
					- i \psi
				\right)
			- \frac i2\psi 
				\left( 
					\psi_t^* \cdot 0
					+ \nabla \psi ^*\cdot 0
					+ i \psi^*
				\right)
			-\mathcal{L}\cdot 0
		\right]
		d\bvec{x} \\
		& = \norm{\psi}_2^2,
\end{align*}
i.e., the power, is a conserved quantity.
Other conservation laws can be found in a similar manner.


\end{document}